%% file: drinker-paper-arxiv.tex
\documentclass[final]{article}
\usepackage[utf8]{inputenc}
\usepackage[T1]{fontenc}
\usepackage[a4paper]{geometry}

\usepackage{listings}
\usepackage{url}

\usepackage{multirow}

\usepackage{mymaths}

\newtheorem{example}[definition]{Example}

\title{The Drinker Paradox and its Dual}

\author{
	Louis Warren and  Hannes Diener and Maarten McKubre-Jordens \\
	Department of Mathematics and Statistics \\
	University of Canterbury
}

\usepackage{logic}
\usepackage{todonotes}
\usepackage{pdflscape}
\allowdisplaybreaks
\begin{document}

\maketitle
%%%%%%%%%%%%%%%%%%%%%%%%%%%%%%%%%%%%%%%%%%%%%%%%%%%%%%%%%%%%%%%%%%%%%%%%%%%%%%%%

\begin{abstract}
The Drinker Paradox is as follows.
\begin{quote}
In every nonempty tavern, there is a person such that if that person is
drinking, then everyone in the tavern is drinking.
\end{quote}
Formally,
\[
	\exists x \big(\varphi \rightarrow \forall y \varphi[x/y]\big)
\ . \]
Due to its counterintuitive nature it is called a paradox, even though it actually is a classical tautology. However, it is not
minimally (or even intuitionistically) provable. The same can be said of its dual, which is (equivalent to) the
well-known principle of \emph{independence of premise},
\[
	\varphi \rightarrow \exists x \psi
	\ \vdash \ \exists x (\varphi \rightarrow \psi)
\]
where $x$ is not free in $\varphi$.

In this paper we study the implications of adding these and other formula
schemata to minimal logic. We show first that these principles are independent of the law of excluded middle and of each other, and second how these schemata relate to
other well-known principles, such as Markov's Principle of unbounded search, providing proofs and semantic models where appropriate.
\end{abstract}

%%%%%%%%%%%%%%%%%%%%%%%%%%%%%%%%%%%%%%%%%%%%%%%%%%%%%%%%%%%%%%%%%%%%%%%%%%%%%%%%

\section{Introduction}
Minimal logic \cite{schwichtenberg}
provides, as its name suggests, a  minimal setting for logical investigations.
Starting from minimal logic, we can get to intuitionistic logic by adding
\emph{ex falso quodlibet} (EFQ), and to classical logic, by adding \emph{double
negation elimination} (DNE).\footnote{Either as a rule, or an axiom scheme. See
below for details.} Therefore, every statement proven over minimal logic can
also be proven in intuitionistic logic and classical logic. In addition, minimal
logic has many structural advantages, and is easier to analyse. Of course, there
is a price one has to pay for working within a weak framework. The price is that
fewer well-known statements are provable outright, which leads to the question
of how they relate. This is a very similar question to the one considered in
constructive reverse mathematics (CRM; \cite{hI06a,hDHabil}), where the aim is to find some ordering in a
multitude of principles, over intuitionistic logic. CRM has been around for some
decades now, and some even trace the origins back to Brouwerian counterexamples. 
Most results in CRM are focused on analysis, where most theorems can be classified into being equivalent to about ten major principles.
It is  a natural question to ask whether we can find similar 
results in the absence of EFQ. Previous work by a subset of the authors
\cite{materialimplications} has investigated the case of propositional schemata,
but has left the predicate case untouched. Similar work can also be found in \cite{jG11,sO08}.
A more detailed approach, but again focused on the propositional case can be found in \cite{hI16}, where it was studied exactly which instances of an axiom scheme are required to prove a given instance of another axiom scheme over minimal logic.
In this paper we will make first steps in the predicate case. As is so often the case, the first-order analysis is
subtler and technically more difficult to deal with.

For the sake of brevity and readability we have only included non-trivial  proofs. The missing proofs, which are in natural deduction style, have been put into an appendix. A version of this paper including that appendix will be made available on \url{arxiv.org} under the same title.

%%%%%%%%%%%%%%%%%%%%%%%%%%%%%%%%%%%%%%%%%%%%%%%%%%%%%%%%%%%%%%%%%%%%%%%%%%%%%%%%

\section{Technical Preliminaries}
  We will generally follow the notation and definitions found in
\cite{schwichtenberg}.

An $n$-ary \emph{scheme} $\mathrm{SCH}(X_1, \dots,X_n)$ is a formula $\mathrm{SCH}$ containing $n$ propositional variables $X_1, \dots,X_n$. An instance  $\mathrm{SCH}(\Phi_1, \dots, \Phi_n)$ is obtained by replacing the variables with formulae $\Phi_1, \dots, \Phi_n $.  A scheme is \emph{derivable} in a logical system
if every instance is derivable in that system. A scheme is
\emph{minimal} (\emph{constructive}) (\emph{classical}) if it is derivable in
minimal (intuitionistic) (classical) logic.

\begin{example}
	The \emph{law of excluded middle}, $\mathrm{LEM}(\Phi) := \Phi \mor
	\lnot\Phi$ is a classical unary scheme.
\end{example}

A logical system can be extended by adding that certain schemata are derivable in
the system. In the case of natural deduction and minimal logic, an extension by
 \LEM/ is an addition of a deduction rule
\begin{prooftree}
	\AxiomC{}
	\RightLabel{$\LEMs(\alpha)$}
	\UnaryInfC{$\alpha \mor \lnot\alpha$}
\end{prooftree}
for every formula $\alpha$. This produces a subsystem of classical logic.

More general, if a formula $\Phi$ is derivable over minimal (intuitionistic) logic
extended by schemata $S_0, S_1, \dotsc, S_n$, then we write
\[
	\derives_{S_0, S_1, \dotsc S_n} \Phi
\]
($\derives_{i , S_0 , S_1 , \dotsc S_n} \Phi$).

Extending a logic by a scheme differs from allowing undischarged assumptions
of instances of the scheme. For example, it should follow from \LEM/ that every
predicate is decidable. Consider the proof of $\derives_{\LEMs} \forallx (Px
\mor \lnot Px)$:
\begin{prooftree}
\AxiomC{}
\RightLabel{LEM}
\UnaryInfC{$P{x} \Tor \Tneg{P{x}}$}
\RightLabel{\Tunivintro}
\UnaryInfC{$\Tforall_{x} \left(P{x} \Tor \Tneg{P{x}}\right)$}
\end{prooftree}
The proof uses $\LEMi{Px}$. However, 
\[ 
\LEMi{Px} \nderives \forallx(Px \mor \lnot Px) 
\ ,\]
since the rule \Tunivintro{} requires that $x$ is not free in any open
assumptions.\footnote{If we defined \LEM/ as the axiom scheme $\forall \vec{x} P\vec{x} \lor \lnot P\vec{x}$, there would be no difference between adding it as a rule or an assumption. This trick is the same as used in \cite[page 14]{schwichtenberg} for \EFQ/ and stability.}

It is trivial to check that the following holds.
\begin{proposition}
Define $\mathrm{DNE}(\Phi) := \lnot\lnot\Phi \mimp \Phi$, and $\EFQs(\Phi) :=
\bot \mimp \Phi$. For all (finite) collections of schemata $S$ and $T$,
\[
	S \derives_i T \iff S \derives_\EFQs T
\]
and
\[
	S \derives_c T \iff S \derives_\DNEs T
\ . \]
\end{proposition}

A preorder $\supset$ may be defined on finite collections of schemata by
considering derivability over extensions of minimal logic.

\begin{definition}
For schemata $S_0, S_1, \dotsc S_n$ and $m$-ary scheme $T$, we write
\[ S_0, S_1, \dotsc S_n \supset T \] if
\[
	\derives_{S_0, S_1, \dotsc, S_n} T(\alpha_0, \dotsc \alpha_m)
\]
for all formulae $\alpha_0, \dotsc, \alpha_n$. We say that $T$ is
\emph{reducible} to $S_0, \dotsc S_n$. Intuitively, a proof using the scheme $T$
can be replaced by a proof using $S_0, S_1, \dotsc S_n$. The relation
`$\supset$' extends to multiple schemata on the right-hand side in the obvious
way.
\end{definition}

To demonstrate that $A_0, A_1, \dotsc A_n \not\supset B$, we exhibit a
Kripke model (see Section 5.3 of \cite{dD04} for more details on Kripke semantics\footnote{While technically speaking the Kripke semantics described in \cite{dD04} are for the intuitionistic case, we can use them in the minimal one, by not forcing and condition on $\bot$---that is treating it just like some fixed propositional symbol.}) in which an instance of $B$ does not hold, but where $A_0, \dotsc
A_n$ hold for every formula. A full model, as described in
\cite{materialimplications}, is sufficient. A full model is one where we can freely create predicates, as long as they satisfy the usual monotonicity requirements. So it is full in the sense that everything that potentially is the interpretation of a predicate actually is one. In Section \ref{Sec:nonfullmodels} we will have to consider non-full models. An \emph{intuitionistic Kripke model} is one where $\bot$ is never forced at any world. These are exactly the Kripke models that force \EFQ/.

In given Kripke diagrams, each state $A$ has its labelled propositions on the
right, and the domain (denoted $T(A)$) on the left. Where the domain is given as
$\naturals$, it should be interpreted as the countable set of constants $\{0, 1,
\dotsc\}$, without the addition of any function terms.

\begin{proposition}
If $\derives_{B_0, \dotsc B_m} \Phi$, and $A_0, \dotsc A_n \supset 
B_0, \dotsc B_m$, then $\derives_{A_0, \dotsc A_n} \Phi$.
\end{proposition}
\begin{proof}
	Consider a natural deduction proof of $\derives_{B_0 + \dotsb + B_m} \Phi$.
	For each $k$, replace each instance of the rule $B_k$ with a proof of
	$\derives_{A_0 + \dotsb + A_n} B_k$. This produces the required derivation.
\end{proof}

We examine relative strengths of a selection of schemata by
considering their relations under `$\supset$'. The renaming of bound variables
in a scheme should not affect its strength. For simplicity of notation, it is
therefore assumed that when working a predicate $Px$, any variables other than
$x$ which appear in quantifiers are bound in $Px$. We write $Py$ as shorthand
for $Px[x/y]$.

%%%%%%%%%%%%%%%%%%%%%%%%%%%%%%%%%%%%%%%%%%%%%%%%%%%%%%%%%%%%%%%%%%%%%%%%%%%%%%%%

\section{Principles}
In addition to \DNE/, \LEM/, and \EFQ/, which are included below for convenience, we examine the following principles as axiom schemata over minimal logic:
\begin{itemize}
\setlength\itemsep{.4em}
	\item[] $\DNEi{A} := \lnot\lnot A \mimp A$\hfill
		(Double Negation Elimination\footnote{Also known as ``Stability''.})
\item[] $\EFQi{A} := \bot \mimp A$\hfill
	(Ex Falso Quodlibet\footnote{Also known as ``explosion''.})
\item[] $\LEMi{A} := A \mor \lnot A$\hfill
	(Law of Excluded Middle\footnote{Also known as the ``principle of
	excluded middle'' and as ``tertium non datur''.})
\item[] $\WLEMi{A} := \lnot A \mor \lnot\lnot A$\\ \strut\hfill
	(Weak Law of Excluded Middle)
\item[] $\DGPi{A, B} := (A \mimp B) \mor (B \mimp A)$\\\strut\hfill
	(Dirk Gently's Principle\footnote{The name \DGP/ was introduced in
	\cite{materialimplications}, and is a literary reference to the novel
	\cite{dA87}, whose main character believes in ``the fundamental
	interconnectedness of all things''. \DGP/ is otherwise also known as (weak)
linearity, and is the basis for G\"odel-Dummett logic \cite{vonplato-gd}.})
\item[] $\DPi{Px} := \existsy (Py \mimp \forallx Px)$\hfill
	(Drinker Paradox)
\item[] $\HEi{Px} := \existsy (\exists x Px \mimp Py)$\\\strut\hfill
	(Schematic Form of Hilbert's Epsilon)
\item[] $\GMPi{Px} := \lnot\forallx Px \mimp \existsx \lnot Px$\\\strut\hfill
	(General Markov's Principle)
\item[] $\GLPOi{Px} := \forallx \lnot Px \mor \existsx Px$\\\strut\hfill
	(General Limited Principle of Omniscience)
\item[] $\GLPOAi{Px} := \forallx Px \mor \existsx \lnot Px$\\\strut\hfill
	(Alternate General Principle of Omniscience)
\item[] $\DNSUi{Px} := \forallx \lnot\lnot Px \mimp \lnot\lnot\forallx
	Px$\\\strut\hfill
	(Universal Double Negation Shift)
\item[] $\DNSEi{Px} := \lnot\lnot \existsx Px \mimp \existsx \lnot\lnot
	Px$\\\strut\hfill
	(Existential Double Negation Shift)
\item[] $\CDi{Px, Q} := \forallx(Px \mor \existsx Q) \mimp \forallx Px \mor
	\existsx Q$\\\strut\hfill
	(Constant Domain)
\item[] $\IPi{Px, Q} := (\existsx Q \mimp \existsx Px) \mimp \existsx(\existsx
	Q \mimp Px)$\\\strut\hfill
	(Independence of Premise)
\end{itemize}

These principles are \emph{all} classically derivable. That is, \DNE/ implies
all of these principles in the sense of $\supset$.

Principles \CD/ and \IP/ are also stated as
\begin{align*}
\CDi{Px, Q} &\equiv \forallx(Px \mor Q) \mimp \forallx Px \mor Q \\
\IPi{Px, Q} &\equiv (Q \mimp \existsx Px) \mimp \existsx(Q \mimp Px)
\end{align*}
where $x$ is not free in $Q$. These forms are syntactically equivalent to the
definitions above for such $Q$, but the variable freedom condition is not
convenient to work with when classifying schemata.

%%%%%%%%%%%%%%%%%%%%%%%%%%%%%%%%%%%%%%%%%%%%%%%%%%%%%%%%%%%%%%%%%%%%%%%%%%%%%%%%

\section{The Drinker Paradox and Hilbert's Epsilon}

The \emph{drinker paradox}, which was popularised by Smullyan in his book of puzzles \cite{rS90}, is the scheme
\[
	\DPs(Px) := \exists_y (Py \mimp \forall_x Px)
\ . \]
Liberally interpreted, it states that (in every nonempty tavern) there exists a person such that if that person is drinking, then everyone (in the tavern) is drinking.

Classically this is true because there is always a \emph{last} person to be
drinking, and it is true for that person. Due to various non-classical
interpretations of ``there is'', however, countermodels may be formed (see
Figure \ref{kripke:two-world-growing-terms}). Notably, the constructivist may object that it is not always clear who is the last to drink---except in the case of a tavern in which the number of patrons is an enumerable positive integer amount.

The drinker paradox can alternatively be stated as
\[
	\exists_y \forall_x (Py \mimp Px)
\ .\footnote{For a proof that this is actually an equivalent formulation see the appendix.} \]

The dual of the drinker paradox is the scheme
\[
	\HEi{Px} := \exists_y (\exists_x Px \mimp Py)
\ , \]
or alternatively,
\[
	\exists_y \forall_x (Px \mimp Py)
\ . \]
\HE/ resembles an axiom scheme form of Hilbert's Epsilon operator \cite{jA16}. In
particular, within a natural deduction proof, from $\exists_x Px$ it allows a
temporary name for a term satisfying $P$ to be introduced. It is equivalent to
\emph{Independence of Premise}
\[
	\IPi{Px, Q} := (\existsx Q \mimp \existsx Px) \mimp \existsx(\existsx Q \mimp Px)
\ . \]
This does not have the same power as Hilbert's Epsilon operator, however.\footnote{Milly Maietti has communicated to us the---currently unpublished---result that Hilbert's Epsilon operator implies the drinker paradox. Thus, together with our results in this paper this shows that the operator version of \HE/ is stronger than the scheme version.}

We will now characterise (full Kripke) models in which \DP/ and/or \HE/ hold,
and use these to separate the two schemata. We will ignore models containing
disconnected states (i.e. models where there are pairs of states such that every
state related to one is unrelated to the other), as these can be examined by the
characteristics of the individual components.

First consider a model with states $A \preceq B$ where there is a term $t \in
T(B) \setminus T(A)$ (for example Figure \ref{kripke:two-world-growing-terms}).
Create a predicate $Px$ with $A \forces Ps$ for all $s \in T(A)$ (and take the
upwards closure). Now $B \not\forces Pt$, so $A \not\forces Ps \mimp \forall_x
Px$, so \DP/ fails. Furthermore, create a predicate $Qx$ with $B \forces Qt$ (and
take the upwards closure). Then $B \forces \exists_x Qx$, but $B \not\forces Qs$
for any $s \in T(A)$. Thus \HE/ fails at $A$. Hence any model for either \DP/ or
\HE/ must have the same terms known at every related pair of states. We will
from now on consider only these models. Moreover, note that a system with only
one term at each state trivially models \DP/ and \HE/.

\begin{figure}[tbp] \centering
\begin{tikzpicture}[kripke]
	\node[world] (a) {$A$};
	\node[world] (b) [above=of a] {$B$};
	\worldlabel{a}{$\{s\}$}{$Ps$}
	\worldlabel{b}{$\{s, t\}$}{$Ps, Qt$}
	\path[->] (a) edge (b);
\end{tikzpicture}
\caption{Kripke countermodel for $\DPi{Px}$ and $\HEi{Qx}$}
\label{kripke:two-world-growing-terms}
\end{figure}

Now consider a model with a branch in it, i.e.\ there are states $A, B, C$
such that $A \preceq B$, $A \preceq C$, and $B$ is not related to $C$. Assume
there are at least two distinct terms understood at $A$. Let $t$ be one such
term. Then create a predicate $P$ with $B \forces Pt$, and $C \forces Ps$ for
all terms $s \in T(A) : s \neq t$ (and any other states forcing these atomic
formulae as required to maintain upwards closure). Certainly neither $B$ nor $C$
force $\forall_x Px$, but for every $u \in T(A)$ either $B$ or $C$ forces $Pu$,
so \DP/ fails at $A$. Furthermore if $u \in T(A)$ then either $B$ or $C$ will
fail to force $Pu$, but both states force $\exists_x Px$, so \HE/ also fails at
$A$. Hence any model for \DP/ or \HE/ with more than two terms must have no
branches, i.e.\ be totally ordered.

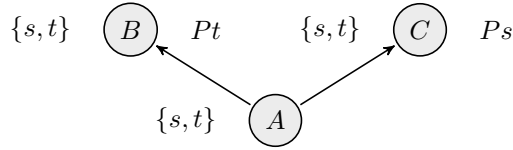
\begin{figure}[H] \centering
\begin{tikzpicture}[kripke]
	\node[world] (a) {$A$};
	\node[world] (b) [above left=0.7cm and 1.4cm of a] {$B$};
	\node[world] (c) [above right=0.7cm and 1.4cm of a] {$C$};
	\worldlabel{a}{$\{s, t\}$}{}
	\worldlabel{b}{$\{s, t\}$}{$Pt$}
	\worldlabel{c}{$\{s, t\}$}{$Ps$}
	\path[->] (a) edge (b);
	\path[->] (a) edge (c);
\end{tikzpicture}
\caption{Kripke countermodel for both $\DPi{Px}$ and $\HEi{Px}$}
\label{kripke:v-const-term}
\end{figure}

Consider then a linear model with finitely many terms. Given a predicate
$Qx$, if every state forces $Qt$ for every term or if every state does not force
$Qt$ for any term, then both \DP/ and \HE/ trivially hold (by applying the
classical reasoning), so we may suppose that this is not the case. For each term $t$, assign a
set $U_t = \{A \in \Sigma | A \not\forces Qt\}$. By upwards closure (and the
assumed linearity), if $t$ and $s$ are terms then either $U_t \subseteq U_s$ or
$U_s \subseteq U_t$, meaning these sets are totally ordered with respect to the
subset relation. There are finitely many of them, so there must be a maximal set
$U_{t_\text{max}}$ with associated term $t_\text{max}$. Suppose a state $A$
forces $Qt_\text{max}$. Then $A \not\in U_{t_\text{max}}$, and so $A \not\in
U_s$ for every term $s$. Thus $A$ forces $Qs$. Hence $Q{t_\text{max}} \mimp
\forall_x Qx$ holds in the model, and so this is a model for \DP/. A similar
argument shows \HE/ also holds, using sets $V_t = \{A \in \Sigma | A \forces
Qt\}$, and in particular the maximal set $V_{t_0}$, to show that $\exists_x Qx
\mimp Q_{t_0}$ is forced everywhere.

We now know that to separate \DP/ and \HE/ we require linear models with
infinitely many terms.

%\begin{center}
%\begin{tikzpicture}[node distance=1 cm, auto, align=center]
%	\node (A) {$A$};
%	\node (Aterms) [left of=A,anchor=east]
%		{$s, t$};
%	\node (Aatom) [right of=A,anchor=west]
%		{};
%	\node (B) [above left of=A] {$C$};
%	\node (Bterms) [left of=B,anchor=east]
%		{$s, t$};
%	\node (Batom) [right of=B,anchor=west]
%		{$Ps$};
%	\node (C) [above right of=A] {$C$};
%	\node (Cterms) [left of=C,anchor=east]
%		{$s, t$};
%	\node (Catom) [right of=C,anchor=west]
%		{$Pt$};
%
%	\draw[->] (A) to node {} (B);
%	\draw[->] (A) to node {} (C);
%\end{tikzpicture}
%\end{center}

\begin{proposition} \label{prop:he--dp}
\HE/ does not imply \DP/ in intuitionistic logic.
\end{proposition}
\begin{proof}
Consider the (intuitionistic) Kripke model with infinitely many worlds below.
In general, $A_n \preceq A_{n + 1}$ and $A_n \forces P0 \dotsc Pn$, and the
domain at every world is $\naturals$.
\hfill
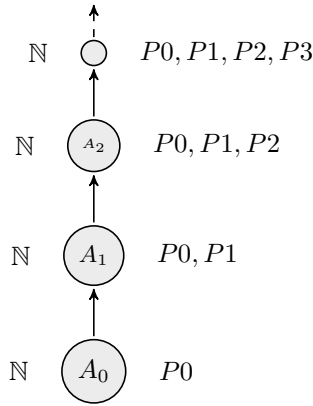
\begin{center}
\begin{tikzpicture}[kripke]
	\node[world] (a0) {$A_0$};
	\node[world] (a1) [above=of a0] {\small $A_1$};
	\node[world] (a2) [above=of a1] {\tiny $A_2$};
	\node[world] (a3) [above=of a2] {};
	\node[vanishpoint] (avanish) [above=of a3] {};
	\worldlabel{a0}{$\naturals$}{$P0$}
	\worldlabel{a1}{$\naturals$}{$P0, P1$}
	\worldlabel{a2}{$\naturals$}{$P0, P1, P2$}
	\worldlabel{a3}{$\naturals$}{$P0, P1, P2, P3$}
	\path[->] (a0) edge (a1);
	\path[->] (a1) edge (a2);
	\path[->] (a2) edge (a3);
	\path[->] (a3) edge[vanisharrow] (avanish);
\end{tikzpicture}
\label{kripke:dp-cm}
\end{center}
\hfill\\
No state forces $\forall_x Px$, but for any term $t \in T(A_0)$ we have $A_0
\preceq A_t$ and $A_t \forces Pt$. Therefore $A_0 \not\forces \exists_y (Py
\mimp \forall_x Px)$, i.e.\ \DP/ does not hold in this model. (In fact, this
argument works for any state.)

Now consider any predicate $Qx$ in this model. If there is no state forcing $Qt$
for some $t \in \naturals$, then trivially every state forces $\exists_x Qx \mimp
Q0$, and it follows that \HE/ is forced. On the other hand, if there are $i, t
\in \naturals$ such that $A_i \forces Qt$, then choose a pair $i, t$ with
minimal $i$. Then, by upwards closure, $\exists_x Qx \mimp Qt$ is forced by
every state.  Hence every state forces \HE/.
\end{proof}

The above model is also a countermodel for \DNSU/. As $\lnot Pt$ is not forced
at any world for any $t$, $A_0 \forces \forallx \lnot\lnot Px$. However $A_0
\forces \lnot\forallx Px$, so $A \not\forces \DNSUi{P}$.

\begin{proposition} \label{prop:dp--he}
\DP/ does not imply \HE/ in intuitionistic logic.
\end{proposition}
\begin{proof}
Consider the (intuitionistic) Kripke system with states $A_{0} \succeq A_{-1}
\succeq A_{-2} \succeq \dotsc \succeq A_{-\infty}$. Let $T(B) = \naturals$ for
every state $B$. Set $F(A_{-\infty}) = \emptyset$, and $F(A_{-n}) = \{Pn,
P(n+1), P(n+2), \dotsc\}$.
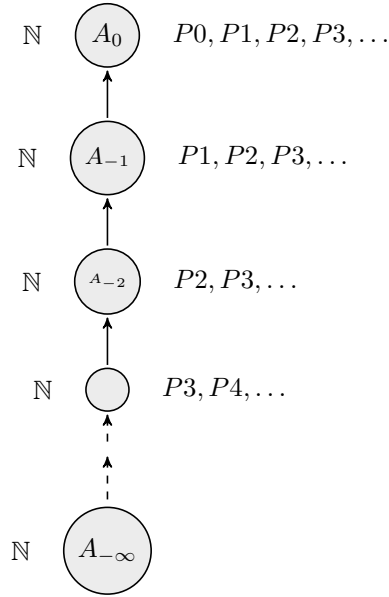
\begin{center}
\begin{tikzpicture}[kripke]
	\node[world] (a0) {$A_0$};
	\node[world] (a1) [below=of a0] {\small $A_{-1}$};
	\node[world] (a2) [below=of a1] {\tiny $A_{-2}$};
	\node[world,inner sep=0.2cm] (a3) [below=of a2] {};
	\node[vanishpoint] (avanish) [below=of a3] {};
	\node[world] (ainf) [below=of avanish] {$A_{-\infty}$};
	\worldlabel{a0}{$\naturals$}{$P0, P1, P2, P3, \dotsc$}
	\worldlabel{a1}{$\naturals$}{$P1, P2, P3, \dotsc$}
	\worldlabel{a2}{$\naturals$}{$P2, P3, \dotsc$}
	\worldlabel{a3}{$\naturals$}{$P3, P4, \dotsc$}
	\worldlabel{ainf}{$\naturals$}{}
	\path[<-] (a0) edge (a1);
	\path[<-] (a1) edge (a2);
	\path[<-] (a2) edge (a3);
	\path[<-] (a3) edge[vanisharrow] (avanish);
	\path[<-] (avanish) edge[vanisharrow] (ainf);
\end{tikzpicture}
\label{kripke:he-cm}
\end{center}
\hfill\\
Let $t \in T(A_{-\infty})$. Then $A_{-(t+1)} \not\forces Pt$. However,
$A_{-(t+1)} \forces P(t+1)$, so $A_{-(t+1)} \forces \exists_x Px$. Therefore
$A_{-(t+1)} \not\forces \exists_x Px \mimp Pt$. Thus $A_{-\infty} \not\forces
\exists_y(\exists_x Px \mimp Pt)$, so \HE/ does not hold in this model.

Now consider any predicate $Qx$ in this model. If every state forces $\forall_x
Qx$, then trivially they also force $\exists_y (Qy \mimp \forall_x Qx)$. On the
other hand, if there are $i, t \in \naturals$ such that $A_{-i} \not\forces Qt$
then choose a pair $i, t$ with minimal $i$ (i.e.\ maximal $A_{-i}$). Then by
upwards closure, whenever $Qt$ is forced, $\forall_x Qx$ is also forced. Hence
every state forces $Qt \mimp \forall_x Qx$, and so also forces $\DP/$.
\end{proof}

In general, if a model contains an infinite sequence of states $A_0 \preceq A_1
\preceq \dotsb$, then a predicate $P$ can be constructed as in Proposition \ref{prop:he--dp}
in order to contradict \DP/. On the other hand if no such sequence exists then
every sequence of related states has a maximal element. Following reasoning in
Proposition \ref{prop:dp--he} shows that \DP/ will hold in such a model.

Conversely, if a model contains an infinite sequence of states $B_0 \succeq
B_{-1} \succeq \dotsb$, along with an element $B_{-\infty}$ which precedes every
state in the sequence, then $P$ may be constructed as in Proposition \ref{prop:dp--he},
contradicting \HE/.

If, on the other hand, no such states exist, then every set of related states
either contains a minimal element or has no lower bound, i.e. every set of
states contains its infimum. Let $A$ be a state in such a model. Now consider
the set $S$ of states above $A$ which force $\exists_x Px$.
If $S = \emptyset$, then vacuously $A \forces \exists_t Px \mimp Pt$ for every
term $t$, so $A$ forces \HE/. Otherwise, note that $A$ is certainly a lower
bound for $S$. By the above assumption, $S$ must have a minimum element $B$. Now
$B \forces \exists_x Px$ so $B \forces Pt$ for some $t$. By upwards closure, $C
\forces Pt$ for every $C \succeq B$, and so specifically for all $C \in S$. Thus
whenever $A \preceq C$ and $C \forces \exists_x Px$, we have $C \in S$, so $C
\forces Pt$. Then $A \forces \exists_x Px \mimp Pt$, and so $A$ forces \HE/.
Hence \HE/ is forced by every state, and so holds in this model.

We now have a characterisation for models of \DP/ and \HE/. They are the models
wherein every state has exactly one term, or otherwise,
\begin{itemize}
	\item the model is linear, and
	\item all terms are known at all states (domain is constant), and
	\item (to model \DP/) every set of states has a maximal element, and/or
	\item (to model \HE/) every set of states contains its infimum.
\end{itemize}
Where $T$ is the set of terms (at every state):
\begin{center}
\begin{tabular}{ |p{2.9cm}|c|c|c| }
	\cline{2-4} \multicolumn{1}{c|}{}
	& $|T| = 1$ & $|T| \in \naturals$ & $|T| \geq |\naturals|$
\\ \hline
Branched & \DP/, \HE/ & Neither & Neither
\\ \hline
Linear & \DP/, \HE/ & \DP/, \HE/ & Indeterminate
\\ \hline
Linear, $\max{\Pi}$ exists for all $\Pi \subset \Sigma$
	& \DP/, \HE/ & \DP/, \HE/ & \DP/
\\ \hline
Linear, $\inf{\Pi} \in \Pi$ for all $\Pi \subset \Sigma$
	& \DP/, \HE/ & \DP/, \HE/ & \HE/
\\ \hline
Both of the two above
	& \DP/, \HE/ & \DP/, \HE/ & \DP/, \HE/
\\ \hline
\end{tabular}
\end{center}

If a model has graph-like connectedness, where all related pairs of states have
finitely many states between them (and so finite paths between them), then it
cannot fall under the third or fourth rows, and so cannot separate \DP/ and \HE/.

 The models are evocative of the intuitions. For, recall the ``last drinker in the tavern'' reason for accepting \DP/ as true; similarly \HE/ can be justified by pointing to ``the first person to drink''.

\begin{corollary} \label{cor:lemdp--lemhilepx}
\DP/ and \HE/ are independent of each other in minimal logic with \LEM/ (and so
certainly over decidable predicates).
\end{corollary}
\begin{proof}
Recall the Kripke systems in Propositions \ref{prop:he--dp} and \ref{prop:dp--he}.
Considering them now as minimal Kripke systems, and forcing $\bot$ at every state
forces \LEM/ everywhere, but their respective separations still hold.
\end{proof}

\section{Separations without full models} \label{Sec:nonfullmodels}

The \emph{Constant Domain} principle is
\[
	\CDi{Px, Q} := \forallx(Px \mor \existsx Q) \mimp \forallx Px \mor \existsx Q
\ . \]
Consider a full Kripke model in which all related worlds have the same domain.
For a world $A$, if $A \forces \forallx(Px \mor \existsx Q)$ then $A
\forces Pt \mor \existsx Q$ for all $t$ in the domain. If $A \not\forces
\existsx Q$, then $A \forces Pt$, and so $A \forces \forallx Px$. Therefore this
is a model for \CD/. Hence any full Kripke countermodel for \CD/ must have
related worlds with different domains, and so must also be a countermodel to
\HE/ (from the section above).

However, we cannot conclude $\HEs \reduces \CDs$, as restriction to full Kripke
models does not preserve completeness of Kripke semantics. To see that
$\nderives_{\HEs} \CDi{Px, Q}$, we require a non-full countermodel to \CD/ in
which \HE/ holds. Therefore, a notion of an axiom scheme holding in a non-full
model is needed. For every formula $\Phi$ in the model, $\HEi{\Phi}$ should be
forced. Formulae in the model should be at least closed with respect to the
logical operations `$\mimp$', `$\mand$', `$\mor$', `$\exists$', and `$\forall$',
and `$\bot$' must also be a formula. The constants in the domain of the root
world may also appear in formulae, but no others.

Consider the following infinite model:
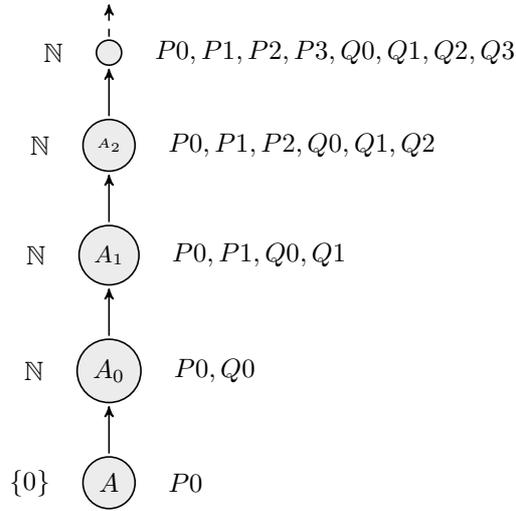
\begin{center}
\begin{tikzpicture}[kripke]
	\node[world] (a)  {$A$};
	\node[world] (a0) [above=of a]  {$A_0$};
	\node[world] (a1) [above=of a0] {\small $A_1$};
	\node[world] (a2) [above=of a1] {\tiny $A_2$};
	\node[world] (a3) [above=of a2] {};
	\node[vanishpoint] (avanish) [above=of a3] {};
	\worldlabel{a}{$\{0\}$}{$P0$}
	\worldlabel{a0}{$\naturals$}{$P0, Q0$}
	\worldlabel{a1}{$\naturals$}{$P0, P1, Q0, Q1$}
	\worldlabel{a2}{$\naturals$}{$P0, P1, P2, Q0, Q1, Q2$}
	\worldlabel{a3}{$\naturals$}{$P0, P1, P2, P3, Q0, Q1, Q2, Q3$}
	\path[->] (a)  edge (a0);
	\path[->] (a0) edge (a1);
	\path[->] (a1) edge (a2);
	\path[->] (a2) edge (a3);
	\path[->] (a3) edge[vanisharrow] (avanish);
\end{tikzpicture}
\end{center}
We have $A \not\forces \CDi{Px, Q}$.

\HE/ holds trivially for propositions. It remains to confirm that \HE/ holds for
all predicates which exist in this model.  Predicates are definable by combining
`$Px$' and `$Qx$', with each other and with propositions, using the binary
logical operations. Clearly, combining a predicate with itself in this manner is
trivial. The propositions available are only $P0$, $Q0$, $\bot$, since
\begin{align*}
	\forallx Px \equiv \bot \\
	\forallx Qx \equiv \bot \\
	\existsx Px \equiv P0 \\
	\existsx Qx \equiv Q0
\end{align*}
and $P0$, $Q0$, $\bot$ are closed under the binary logical operations (with
respect to equivalence in this model). First,
\begin{align*}
	Px \mimp Qx \equiv Qx \\
	Qx \mimp Px \equiv P0 \\
	Px \mor Qx  \equiv Px \\
	Px \mand Qx  \equiv Qx
\end{align*}
Now, with $P0$,
\begin{align*}
	Px \mimp P0 \equiv P0 \\
	P0 \mimp Px \equiv Px \\
	Px \mor  P0  \equiv P0 \\
	Px \mand P0  \equiv Px \\
	Qx \mimp Q0 \equiv P0 \\
	Q0 \mimp Qx \equiv Qx \\
	Qx \mor  Q0  \equiv Q0 \\
	Qx \mand Q0  \equiv Qx
. \end{align*}
With $Q0$,
\begin{align*}
	Px \mimp Q0 \equiv Q0 \\
	Q0 \mimp Px \equiv Qx \\
	Px \mor  Q0  \equiv P0 \\
	Px \mand Q0  \equiv Qx \\
	Qx \mimp Q0 \equiv P0 \\
	Q0 \mimp Qx \equiv Qx \\
	Qx \mor  Q0  \equiv Q0 \\
	Qx \mand Q0  \equiv Qx
\end{align*}
Finally, with $\bot$,
\begin{align*}
	Px \mimp \bot \equiv \bot \\
	\bot \mimp Px \equiv P0 \\
	Px \mor  \bot  \equiv Px \\
	Px \mand \bot  \equiv \bot \\
	Qx \mimp \bot \equiv \bot \\
	\bot \mimp Qx \equiv P0 \\
	Qx \mor  \bot  \equiv Qx \\
	Qx \mand \bot  \equiv \bot
.\end{align*}

Thus, $Px$ and $Qx$ really are the only predicates in this model. $A \forces
\HEi{P}, \HEi{Q}$, so we have a non-full model for \HE/ where \CD/ fails.

%%%%%%%%%%%%%%%%%%%%%%%%%%%%%%%%%%%%%%%%%%%%%%%%%%%%%%%%%%%%%%%%%%%%%%%%%%%%%%%%

\section{From first-order to propositional schemata}
Some first-order schemata are infinitary forms of propositional schemata. Viewing
universal and existential generalisation as conjunction and disjunction on
propositional symbols $A$ and $B$, the drinker paradox becomes
\[
	(A \mimp (A \mand B)) \mor (B \mimp (A \mand B))
, \]
and so \DGP/ follows. A formal proof requires embedding $A$ and $B$ in a single
predicate. For example, over the domain of natural numbers, a predicate $P$
such that
\begin{align*}
	P(0)  &\miff A \\
	P(Sn) &\miff B
\end{align*}
gives $\DPi{Px} \derives \DGPi{A, B}$. However, such an embedding is not
possible if the domain contains a single element. It was shown above that \DP/
holds in models with branches if the domain contains only one term, while in
\cite{materialimplications} it is shown that \DGP/ holds only in v-free models.
Therefore there can be no way of deriving instances of \DGP/ from \DP/ without
an embedding using two or more elements in the domain.

\begin{figure}[H] \centering
\begin{tikzpicture}[kripke]
	\node[world] (a) {$X$};
	\node[world] (b) [above left=0.7cm and 1.4cm of a] {$Y$};
	\node[world] (c) [above right=0.7cm and 1.4cm of a] {$Z$};
	\worldlabel{a}{$\{s\}$}{}
	\worldlabel{b}{$\{s\}$}{$A$}
	\worldlabel{c}{$\{s\}$}{$B$}
	\path[->] (a) edge (b);
	\path[->] (a) edge (c);
\end{tikzpicture}
\caption{Kripke countermodel for $\DGPi{A, B}$ where \DP/ holds}
\label{kripke:v-one-term}
\end{figure}
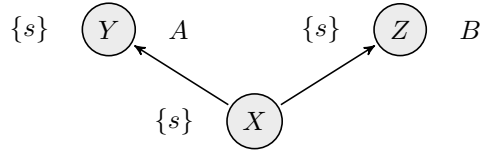

Domain is a semantic concept. In order to derive an instance of \DGP/ using
\DP/, we require syntax corresponding to the existence of more than one
(distinct) term.

\begin{definition}
Natural deduction can be extended by adding term names $0$ and $1$, a unary
predicate $D$, and the rules

\DZ/: \vspace{-\baselineskip}
\begin{prooftree}
\AxiomC{}
\RightLabel{D0}
\UnaryInfC{$D{0}$}
\end{prooftree}

\DO/: \vspace{-\baselineskip}
\begin{prooftree}
\AxiomC{}
\RightLabel{$\Tneg{D1}$}
\UnaryInfC{$\Tneg{D{1}}$}
\end{prooftree}

\DX/: \vspace{-\baselineskip}
\begin{prooftree}
\AxiomC{}
\RightLabel{Dx}
\UnaryInfC{$\Tforall_{x} \left(D{x} \Tor \Tneg{D{x}}\right)$}
\end{prooftree}

$D$ serves to make a weak distinction between the constants named by $0$ and
$1$. \footnote{Bell in \cite{bellepsilon} suggests this ``modest `decidability'
condition" in the form of a decidable equality for a single constant $a$, along
with a constant $b \neq a$.}

We call minimal (intuitionistic) logic extended by these rules \emph{two-termed}
minimal (intuitionistic) logic, in which case we write `$\derives_{\TT/}$' in
place of `$\derives$'.
\end{definition}

Semantically, an intuitionistic Kripke model for \TT/ is one in which
there are two constants $0$ and $1$, $D0$ holds at every world, and $Dn$ is not
forced for $n \neq 0$. For minimal Kripke models, it is also possible instead
that there is only one term, and $\bot$ holds everywhere.

In general, given propositional symbols $A$ and $B$, we want to define a
predicate $P$ such that $\forallx Px \derives A \mand B$ and $\existsx Px
\derives A \mor B$.

We recover
\[
	\DPi{(Dx \mimp A) \mand (\lnot Dx \mimp B)}
		\derives_{\EFQ/, \TT/} \DGPi{A, B}
\]
\[
	\HEi{(Dx \mimp A) \mand (\lnot Dx \mimp B)}
		\derives_{\EFQ/, \TT/} \DGPi{A, B}
\]
\[
	\DPi{(Dx \mimp \lnot\lnot A) \mand (\lnot Dx \mimp \lnot A)}
		\derives_{\TT/} \WLEMi{A}
\]
\[
	\HEi{(Dx \mimp \lnot\lnot A) \mand (\lnot Dx \mimp \lnot A)}
		\derives_{\TT/} \WLEMi{A}
\]
\[
	\GMPi{(Dx \mimp \lnot\lnot A) \mand (\lnot Dx \mimp \lnot A)}
		\derives_{\TT/} \WLEMi{A}
\]
\[
	\DNSEi{(Dx \mimp \lnot\lnot A) \mand (\lnot Dx \mimp \lnot A)}
		\derives_{\TT/} \WLEMi{A}
. \]

%%%%%%%%%%%%%%%%%%%%%%%%%%%%%%%%%%%%%%%%%%%%%%%%%%%%%%%%%%%%%%%%%%%%%%%%%%%%%%%%
\section{Hierarchy}\label{Sec:hierarchy}

The preorder from `$\reduces$` produces a hierarchy. Arrows labelled with
schemes indicate that those schemes must be taken together with the scheme at
the tail to produce the scheme at the head.

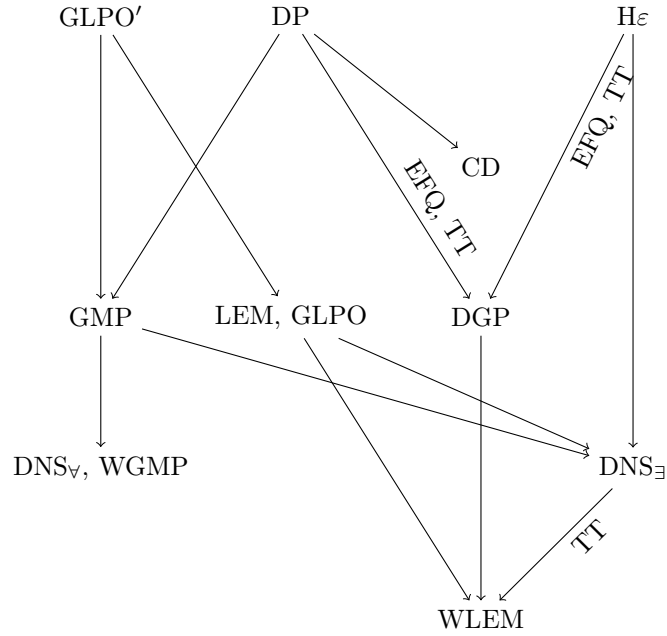
\begin{center}
\begin{tikzpicture}[auto]
\node (glpoa) at (0, 0) {\GLPOA/};
\node (dp) at   (2.5, 0) {\DP/};
\node (he) at   (7,  0) {\HE/};
\node (ud) at   (5, -2) {\CD/};
\node (gmp) at  (0, -4) {\GMP/};
\node (lem) at  (2.5, -4) {\LEM/, \GLPO/};
\node (dgp) at  (5, -4) {\DGP/};
\node (dnsu) at (0, -6) {\DNSU/, \WGMP/};
\node (dnse) at (7, -6) {\DNSE/};
\node (wlem) at (5, -8) {\WLEM/};
\draw[->] (dgp) to[] (wlem);
\draw[->] (dnse) to[] node[midway, sloped] {\TT/} (wlem);
\draw[->] (dp) to[] node[midway, sloped] {\EFQ/, \TT/} (dgp);
\draw[->] (dp) to[] (gmp);
\draw[->] (dp) to[] (ud);
\draw[->] (glpoa) to[] (gmp);
\draw[->] (glpoa) to[] (lem);
\draw[->] (gmp) to[] (dnse);
\draw[->] (gmp) to[] (dnsu);
\draw[->] (he) to[] (dnse);
\draw[->] (he) to[] node[midway, sloped] {\EFQ/, \TT/} (dgp);
\draw[->] (lem) to[] (dnse);
\draw[->] (lem) to[] (wlem);
\end{tikzpicture}
\end{center}

This hierarchy is complete in the sense that no other unlabelled arrows may be
added (see below). Moreover, for arrows labelled with at least one of \EFQ/,
\TT/, the remaining open questions are if $\GMPs, \EFQs \reduces \CDs$ and/or
$\GMPs, \EFQs, \TTs \reduces \CDs$.

\section{Semantics}
In addition to the Kripke model analysis presented earlier, the following full
models give all possible separations of the schemes under investigation. In
cases where models should have \TT/, we omit labelling $D0$ on every world for
the sake of brevity.

In \cite{materialimplications}, it is shown that \DGP/ and \WLEM/ hold in all
v-free models, \EFQ/ holds in a model if and only if $\bot$ is not forced
anywhere, and \LEM/ holds if only one world does not force $\bot$. Revisiting
the countermodels (and previously given reasoning) for \DP/ and \HE/, we have

% dp-cm
\begin{center}
\begin{tikzpicture}[kripke]
	\node[world] (a0) {$A_0$};
	\node[world] (a1) [above=of a0] {\small $A_1$};
	\node[world] (a2) [above=of a1] {\tiny $A_2$};
	\node[world] (a3) [above=of a2] {};
	\node[vanishpoint] (avanish) [above=of a3] {};
	\worldlabel{a0}{$\naturals$}{}
	\worldlabel{a1}{$\naturals$}{}
	\worldlabel{a2}{$\naturals$}{}
	\worldlabel{a3}{$\naturals$}{}
	\path[->] (a0) edge (a1);
	\path[->] (a1) edge (a2);
	\path[->] (a2) edge (a3);
	\path[->] (a3) edge[vanisharrow] (avanish);
\end{tikzpicture}
\end{center}
is a model for \EFQ/, \TT/, \HE/, \DGP/, \WLEM/, \CD/, and a countermodel for
\DP/, \LEM/, \DNSU/, while

% he-cm
\begin{center}
\begin{tikzpicture}[kripke]
	\node[world] (a0) {$A_0$};
	\node[world] (a1) [below=of a0] {\small $A_{-1}$};
	\node[world] (a2) [below=of a1] {\tiny $A_{-2}$};
	\node[world,inner sep=0.2cm] (a3) [below=of a2] {};
	\node[vanishpoint] (avanish) [below=of a3] {};
	\node[world] (ainf) [below=of avanish] {$A_{-\infty}$};
	\worldlabel{a0}{$\naturals$}{}
	\worldlabel{a1}{$\naturals$}{}
	\worldlabel{a2}{$\naturals$}{}
	\worldlabel{a3}{$\naturals$}{}
	\worldlabel{ainf}{$\naturals$}{}
	\path[<-] (a0) edge (a1);
	\path[<-] (a1) edge (a2);
	\path[<-] (a2) edge (a3);
	\path[<-] (a3) edge[vanisharrow] (avanish);
	\path[<-] (avanish) edge[vanisharrow] (ainf);
\end{tikzpicture}
\end{center}
is a model for \EFQ/, \TT/, \DP/, \DGP/, \WLEM/ and a countermodel for \HE/,
\LEM/.

It is trivial to check that model presented in Section \ref{Sec:nonfullmodels}
can be modified as follows, to model both \HE/ and \TT/ while still being a
countermodel to \CD/.
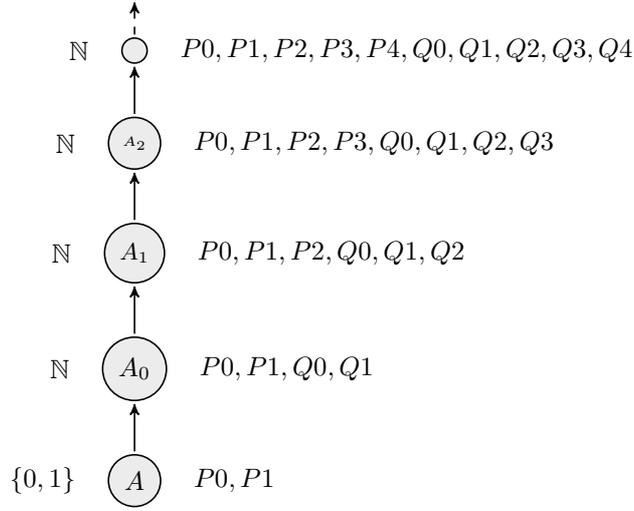
\begin{center}
\begin{tikzpicture}[kripke]
	\node[world] (a)  {$A$};
	\node[world] (a0) [above=of a]  {$A_0$};
	\node[world] (a1) [above=of a0] {\small $A_1$};
	\node[world] (a2) [above=of a1] {\tiny $A_2$};
	\node[world] (a3) [above=of a2] {};
	\node[vanishpoint] (avanish) [above=of a3] {};
	\worldlabel{a}{$\{0, 1\}$}{$P0, P1$}
	\worldlabel{a0}{$\naturals$}{$P0, P1, Q0, Q1$}
	\worldlabel{a1}{$\naturals$}{$P0, P1, P2, Q0, Q1, Q2$}
	\worldlabel{a2}{$\naturals$}{$P0, P1, P2, P3, Q0, Q1, Q2, Q3$}
	\worldlabel{a3}{$\naturals$}{$P0, P1, P2, P3, P4, Q0, Q1, Q2, Q3, Q4$}
	\path[->] (a)  edge (a0);
	\path[->] (a0) edge (a1);
	\path[->] (a1) edge (a2);
	\path[->] (a2) edge (a3);
	\path[->] (a3) edge[vanisharrow] (avanish);
\end{tikzpicture}
\end{center}

It is straightforward to check whether a scheme holds or fails in a given finite
full model; as only (few and) finitely many upwards closed labellings of worlds
are possible, and these may be checked exhaustively. We therefore present the
remaining models without comment.
\input{models}

\bibliographystyle{abbrv}
\bibliography{logicbib}
%%%%%%%%%%%%%%%%%%%%%%%%%%%%%%%%%%%%%%%%%%%%%%%%%%%%%%%%%%%%%%%%%%%%%%%%%%%%%%%%

\begin{landscape}
\include{appendix}

\end{landscape}

\end{document}

%% file: models.tex
%% dp-cm
%\begin{tikzpicture}[kripke]
%	\node[world] (a0) {$A_0$};
%	\node[world] (a1) [above=of a0] {\small $A_1$};
%	\node[world] (a2) [above=of a1] {\tiny $A_2$};
%	\node[world] (a3) [above=of a2] {};
%	\node[vanishpoint] (avanish) [above=of a3] {};
%	\worldlabel{a0}{$\naturals$}{}
%	\worldlabel{a1}{$\naturals$}{}
%	\worldlabel{a2}{$\naturals$}{}
%	\worldlabel{a3}{$\naturals$}{}
%	\path[->] (a0) edge (a1);
%	\path[->] (a1) edge (a2);
%	\path[->] (a2) edge (a3);
%	\path[->] (a3) edge[vanisharrow] (avanish);
%\end{tikzpicture}
%\EFQ/, \TT/, \HE/, \DGP/, \WLEM/, \CD/
%\DP/, \LEM/, \DNSU/
%
%% he-cm
%\begin{tikzpicture}[kripke]
%	\node[world] (a0) {$A_0$};
%	\node[world] (a1) [below=of a0] {\small $A_{-1}$};
%	\node[world] (a2) [below=of a1] {\tiny $A_{-2}$};
%	\node[world,inner sep=0.2cm] (a3) [below=of a2] {};
%	\node[vanishpoint] (avanish) [below=of a3] {};
%	\node[world] (ainf) [below=of avanish] {$A_{-\infty}$};
%	\worldlabel{a0}{$\naturals$}{}
%	\worldlabel{a1}{$\naturals$}{}
%	\worldlabel{a2}{$\naturals$}{}
%	\worldlabel{a3}{$\naturals$}{}
%	\worldlabel{ainf}{$\naturals$}{}
%	\path[<-] (a0) edge (a1);
%	\path[<-] (a1) edge (a2);
%	\path[<-] (a2) edge (a3);
%	\path[<-] (a3) edge[vanisharrow] (avanish);
%	\path[<-] (avanish) edge[vanisharrow] (ainf);
%\end{tikzpicture}
%\EFQ/, \TT/, \DP/, \DGP/, \WLEM/
%\HE/, \LEM/

% v-const
\begin{center}
\begin{tikzpicture}[kripke]
	\node[world] (a) {$A$};
	\node[world] (b) [above left=0.7cm and 1.4cm of a] {$B$};
	\node[world] (c) [above right=0.7cm and 1.4cm of a] {$C$};
	\worldlabel{a}{$\{s, t\}$}{}
	\worldlabel{b}{$\{s, t\}$}{}
	\worldlabel{c}{$\{s, t\}$}{}
	\path[->] (a) edge (b);
	\path[->] (a) edge (c);
\end{tikzpicture}
\end{center}
is a model for \EFQ/, \TT/, \DNSU/, \CD/
and a countermodel for \DP/, \HE/, \DGP/, \WLEM/, \DNSE/.

% v-const-lem
\begin{center}
\begin{tikzpicture}[kripke]
	\node[world] (a) {$A$};
	\node[world] (b) [above left=0.7cm and 1.4cm of a] {$B$};
	\node[world] (c) [above right=0.7cm and 1.4cm of a] {$C$};
	\worldlabel{a}{$\{s, t\}$}{}
	\worldlabel{b}{$\{s, t\}$}{$\bot$}
	\worldlabel{c}{$\{s, t\}$}{$\bot$}
	\path[->] (a) edge (b);
	\path[->] (a) edge (c);
\end{tikzpicture}
\end{center}
is a model for \GLPOA/, \LEM/
and a countermodel for \DP/, \HE/, \DGP/.

% v-one-term
\begin{center}
\begin{tikzpicture}[kripke]
	\node[world] (a) {$A$};
	\node[world] (b) [above left=0.7cm and 1.4cm of a] {$B$};
	\node[world] (c) [above right=0.7cm and 1.4cm of a] {$C$};
	\worldlabel{a}{$\{s\}$}{}
	\worldlabel{b}{$\{s\}$}{}
	\worldlabel{c}{$\{s\}$}{}
	\path[->] (a) edge (b);
	\path[->] (a) edge (c);
\end{tikzpicture}
\end{center}
is a model for \EFQ/, \DP/, \HE/
and a countermodel for \DGP/, \WLEM/.

% v-one-term-lobot
\begin{center}
\begin{tikzpicture}[kripke]
	\node[world] (a) {$A$};
	\node[world] (b) [above left=0.7cm and 1.4cm of a] {$B$};
	\node[world] (c) [above right=0.7cm and 1.4cm of a] {$C$};
	\worldlabel{a}{$\{s\}$}{$\bot$}
	\worldlabel{b}{$\{s\}$}{$\bot$}
	\worldlabel{c}{$\{s\}$}{$\bot$}
	\path[->] (a) edge (b);
	\path[->] (a) edge (c);
\end{tikzpicture}
\end{center}
is a model for \TT/, \DP/, \HE/, \GLPOA/
and a countermodel for \DGP/.

% diamond-const
\begin{center}
\begin{tikzpicture}[kripke]
	\node[world] (a) {$A$};
	\node[world] (b) [above left=0.7cm and 1.4cm of a] {$B$};
	\node[world] (c) [above right=0.7cm and 1.4cm of a] {$C$};
	\node[world] (d) [above left=0.7cm and 1.4cm of c] {$D$};
	\worldlabel{a}{$\{s, t\}$}{}
	\worldlabel{b}{$\{s, t\}$}{}
	\worldlabel{c}{$\{s, t\}$}{}
	\worldlabel{d}{$\{s, t\}$}{}
	\path[->] (a) edge (b);
	\path[->] (a) edge (c);
	\path[->] (b) edge (d);
	\path[->] (c) edge (d);
\end{tikzpicture}
\end{center}
is a model for \EFQ/, \TT/, \WLEM/, \GMP/
and a countermodel for \DP/, \HE/, \DGP/.

% one-world-one-term
\begin{center}
\begin{tikzpicture}[kripke]
	\node[world] (a) {$A$};
	\worldlabel{a}{$\{s\}$}{}
\end{tikzpicture}
\end{center}
is a model for \LEM/, \WLEM/, \DGP/, \GLPOA/, \GMP/, \DP/, \HE/, \DNSU/, \DNSE/, \CD/, \EFQ/
and a countermodel for \TT/.

% two-world-growing-terms-2-3
\begin{center}
\begin{tikzpicture}[kripke]
	\node[world] (a) {$A$};
	\node[world] (b) [above=of a] {$B$};
	\worldlabel{a}{$\{0, 1\}$}{}
	\worldlabel{b}{$\{0, 1, 2\}$}{}
	\path[->] (a) edge (b);
\end{tikzpicture}
\end{center}
is a model for \TT/, \EFQ/, \DGP/, \WLEM/, \DNSU/
and a countermodel for \DNSE/, \CD/.

% two-world-growing-terms-2-3-lem
\begin{center}
\begin{tikzpicture}[kripke]
	\node[world] (a) {$A$};
	\node[world] (b) [above=of a] {$B$};
	\worldlabel{a}{$\{0, 1\}$}{}
	\worldlabel{b}{$\{0, 1, 2\}$}{$\bot$}
	\path[->] (a) edge (b);
\end{tikzpicture}
\end{center}
is a model for \TT/, \LEM/
and a countermodel for \EFQ/, \GMP/, \CD/, \DNSU/, \DP/, \HE/.

% two-world-growing-terms-lobot
\begin{center}
\begin{tikzpicture}[kripke]
	\node[world] (a) {$A$};
	\node[world] (b) [above=of a] {$B$};
	\worldlabel{a}{$\{0\}$}{$\bot$}
	\worldlabel{b}{$\{0, 1\}$}{$\bot$}
	\path[->] (a) edge (b);
\end{tikzpicture}
\end{center}
is a model for \TT/, \DGP/, \GMP/, \GLPOA/
and a countermodel for \EFQ/, \CD/, \HE/, \DP/.

%nonfull already presented

%% file: appendix.tex
\section*{Appendix}

\begin{proposition} \label{prop:DNE->LEM}
$\DNEs \reduces \LEMs$
\begin{proof}
\begin{deduction}
\AxiomC{}
\RightLabel{DNE}
\UnaryInfC{$\Tneg{\Tneg{\left(A \Tor \Tneg{A}\right)}} \Tarrow \left(A \Tor \Tneg{A}\right)$}
	\AxiomC{}
	\UnaryInfC{$\Tneg{\left(A \Tor \Tneg{A}\right)}$}
		\AxiomC{}
		\UnaryInfC{$\Tneg{\left(A \Tor \Tneg{A}\right)}$}
			\AxiomC{}
			\UnaryInfC{$A$}
			\RightLabel{\Tdisjintro}
			\UnaryInfC{$A \Tor \Tneg{A}$}
		\RightLabel{\Tarrowelim}
		\BinaryInfC{$\bot$}
		\RightLabel{\Tarrowintro}
		\UnaryInfC{$\Tneg{A}$}
		\RightLabel{\Tdisjintro}
		\UnaryInfC{$A \Tor \Tneg{A}$}
	\RightLabel{\Tarrowelim}
	\BinaryInfC{$\bot$}
	\RightLabel{\Tarrowintro}
	\UnaryInfC{$\Tneg{\Tneg{\left(A \Tor \Tneg{A}\right)}}$}
\RightLabel{\Tarrowelim}
\BinaryInfC{$A \Tor \Tneg{A}$}
\end{deduction}
\end{proof}
\end{proposition}

\begin{proposition} \label{prop:DNE->EFQ}
$\DNEs \reduces \EFQs$
\begin{proof}
\begin{deduction}
\AxiomC{}
\RightLabel{DNE}
\UnaryInfC{$\Tneg{\Tneg{A}} \Tarrow A$}
	\AxiomC{}
	\UnaryInfC{$\bot$}
	\RightLabel{\Tarrowintro}
	\UnaryInfC{$\Tneg{\Tneg{A}}$}
\RightLabel{\Tarrowelim}
\BinaryInfC{$A$}
\RightLabel{\Tarrowintro}
\UnaryInfC{$\bot \Tarrow A$}
\end{deduction}
\end{proof}
\end{proposition}

\begin{proposition} \label{prop:LEM,EFQ->DNE}
$\LEMs, \EFQs \reduces \DNEs$
\begin{proof}
\begin{deduction}
\AxiomC{}
\RightLabel{LEM}
\UnaryInfC{$A \Tor \Tneg{A}$}
	\AxiomC{}
	\UnaryInfC{$A$}
		\AxiomC{}
		\RightLabel{EFQ}
		\UnaryInfC{$\bot \Tarrow A$}
			\AxiomC{}
			\UnaryInfC{$\Tneg{\Tneg{A}}$}
				\AxiomC{}
				\UnaryInfC{$\Tneg{A}$}
			\RightLabel{\Tarrowelim}
			\BinaryInfC{$\bot$}
		\RightLabel{\Tarrowelim}
		\BinaryInfC{$A$}
\RightLabel{\Tdisjelim}
\TrinaryInfC{$A$}
\RightLabel{\Tarrowintro}
\UnaryInfC{$\Tneg{\Tneg{A}} \Tarrow A$}
\end{deduction}
\end{proof}
\end{proposition}

\clearpage

\begin{proposition} \label{prop:HE->IP}
$\HEs \reduces \IPs$
\begin{proof}
\begin{deduction}
\AxiomC{}
\RightLabel{HE}
\UnaryInfC{$\Texists_{y} \left(\Texists_{x} P{x} \Tarrow P{y}\right)$}
	\AxiomC{}
	\UnaryInfC{$\Texists_{x} P{x} \Tarrow P{y}$}
		\AxiomC{}
		\UnaryInfC{$\Texists_{x} A \Tarrow \Texists_{x} P{x}$}
			\AxiomC{}
			\UnaryInfC{$\Texists_{x} A$}
		\RightLabel{\Tarrowelim}
		\BinaryInfC{$\Texists_{x} P{x}$}
	\RightLabel{\Tarrowelim}
	\BinaryInfC{$P{y}$}
	\RightLabel{\Tarrowintro}
	\UnaryInfC{$\Texists_{x} A \Tarrow P{y}$}
	\RightLabel{\Texistintro}
	\UnaryInfC{$\Texists_{x} \left(\Texists_{x} A \Tarrow P{x}\right)$}
\RightLabel{\Texistelim}
\BinaryInfC{$\Texists_{x} \left(\Texists_{x} A \Tarrow P{x}\right)$}
\RightLabel{\Tarrowintro}
\UnaryInfC{$\left(\Texists_{x} A \Tarrow \Texists_{x} P{x}\right) \Tarrow \Texists_{x} \left(\Texists_{x} A \Tarrow P{x}\right)$}
\end{deduction}
\end{proof}
\end{proposition}

\begin{proposition} \label{prop:IP->HE}
$\IPs \reduces \HEs$
\begin{proof}
\begin{deduction}
\AxiomC{}
\RightLabel{IP}
\UnaryInfC{$\left(\Texists_{x} P{x} \Tarrow \Texists_{x} P{x}\right) \Tarrow \Texists_{x} \left(\Texists_{x} P{x} \Tarrow P{x}\right)$}
	\AxiomC{}
	\UnaryInfC{$\Texists_{x} P{x}$}
	\RightLabel{\Tarrowintro}
	\UnaryInfC{$\Texists_{x} P{x} \Tarrow \Texists_{x} P{x}$}
\RightLabel{\Tarrowelim}
\BinaryInfC{$\Texists_{x} \left(\Texists_{x} P{x} \Tarrow P{x}\right)$}
	\AxiomC{}
	\UnaryInfC{$\Texists_{x} P{x} \Tarrow P{x}$}
	\RightLabel{\Texistintro}
	\UnaryInfC{$\Texists_{y} \left(\Texists_{x} P{x} \Tarrow P{y}\right)$}
\RightLabel{\Texistelim}
\BinaryInfC{$\Texists_{y} \left(\Texists_{x} P{x} \Tarrow P{y}\right)$}
\end{deduction}
\end{proof}
\end{proposition}

\begin{proposition} \label{prop:LEM->GLPO}
$\LEMs \reduces \GLPOs$
\begin{proof}
\begin{deduction}
\AxiomC{}
\RightLabel{LEM}
\UnaryInfC{$\Texists_{x} P{x} \Tor \Tneg{\Texists_{x} P{x}}$}
	\AxiomC{}
	\UnaryInfC{$\Texists_{x} P{x}$}
	\RightLabel{\Tdisjintro}
	\UnaryInfC{$\Tforall_{x} \Tneg{P{x}} \Tor \Texists_{x} P{x}$}
		\AxiomC{}
		\UnaryInfC{$\Tneg{\Texists_{x} P{x}}$}
			\AxiomC{}
			\UnaryInfC{$P{x}$}
			\RightLabel{\Texistintro}
			\UnaryInfC{$\Texists_{x} P{x}$}
		\RightLabel{\Tarrowelim}
		\BinaryInfC{$\bot$}
		\RightLabel{\Tarrowintro}
		\UnaryInfC{$\Tneg{P{x}}$}
		\RightLabel{\Tunivintro}
		\UnaryInfC{$\Tforall_{x} \Tneg{P{x}}$}
		\RightLabel{\Tdisjintro}
		\UnaryInfC{$\Tforall_{x} \Tneg{P{x}} \Tor \Texists_{x} P{x}$}
\RightLabel{\Tdisjelim}
\TrinaryInfC{$\Tforall_{x} \Tneg{P{x}} \Tor \Texists_{x} P{x}$}
\end{deduction}
\end{proof}
\end{proposition}

\begin{proposition} \label{prop:GLPO->LEM}
$\GLPOs \reduces \LEMs$
\begin{proof}
\begin{deduction}
\AxiomC{}
\RightLabel{GLPO}
\UnaryInfC{$\Tforall_{x} \Tneg{A} \Tor \Texists_{x} A$}
	\AxiomC{}
	\UnaryInfC{$\Tforall_{x} \Tneg{A}$}
	\RightLabel{\Tunivelim}
	\UnaryInfC{$\Tneg{A}$}
	\RightLabel{\Tdisjintro}
	\UnaryInfC{$A \Tor \Tneg{A}$}
		\AxiomC{}
		\UnaryInfC{$\Texists_{x} A$}
			\AxiomC{}
			\UnaryInfC{$A$}
		\RightLabel{\Texistelim}
		\BinaryInfC{$A$}
		\RightLabel{\Tdisjintro}
		\UnaryInfC{$A \Tor \Tneg{A}$}
\RightLabel{\Tdisjelim}
\TrinaryInfC{$A \Tor \Tneg{A}$}
\end{deduction}
\end{proof}
\end{proposition}

\clearpage

\begin{proposition} \label{prop:DNSU->WGMP}
$\DNSUs \reduces \WGMPs$
\begin{proof}
\begin{deduction}
\AxiomC{}
\RightLabel{DNSU}
\UnaryInfC{$\Tforall_{x} \Tneg{\Tneg{P{x}}} \Tarrow \Tneg{\Tneg{\Tforall_{x} P{x}}}$}
	\AxiomC{}
	\UnaryInfC{$\Tneg{\Texists_{x} \Tneg{P{x}}}$}
		\AxiomC{}
		\UnaryInfC{$\Tneg{P{x}}$}
		\RightLabel{\Texistintro}
		\UnaryInfC{$\Texists_{x} \Tneg{P{x}}$}
	\RightLabel{\Tarrowelim}
	\BinaryInfC{$\bot$}
	\RightLabel{\Tarrowintro}
	\UnaryInfC{$\Tneg{\Tneg{P{x}}}$}
	\RightLabel{\Tunivintro}
	\UnaryInfC{$\Tforall_{x} \Tneg{\Tneg{P{x}}}$}
\RightLabel{\Tarrowelim}
\BinaryInfC{$\Tneg{\Tneg{\Tforall_{x} P{x}}}$}
	\AxiomC{}
	\UnaryInfC{$\Tneg{\Tforall_{x} P{x}}$}
\RightLabel{\Tarrowelim}
\BinaryInfC{$\bot$}
\RightLabel{\Tarrowintro}
\UnaryInfC{$\Tneg{\Tneg{\Texists_{x} \Tneg{P{x}}}}$}
\RightLabel{\Tarrowintro}
\UnaryInfC{$\Tneg{\Tforall_{x} P{x}} \Tarrow \Tneg{\Tneg{\Texists_{x} \Tneg{P{x}}}}$}
\end{deduction}
\end{proof}
\end{proposition}

\begin{proposition} \label{prop:WGMP->DNSU}
$\WGMPs \reduces \DNSUs$
\begin{proof}
\begin{deduction}
\AxiomC{}
\RightLabel{WGMP}
\UnaryInfC{$\Tneg{\Tforall_{x} P{x}} \Tarrow \Tneg{\Tneg{\Texists_{x} \Tneg{P{x}}}}$}
	\AxiomC{}
	\UnaryInfC{$\Tneg{\Tforall_{x} P{x}}$}
\RightLabel{\Tarrowelim}
\BinaryInfC{$\Tneg{\Tneg{\Texists_{x} \Tneg{P{x}}}}$}
	\AxiomC{}
	\UnaryInfC{$\Texists_{x} \Tneg{P{x}}$}
		\AxiomC{}
		\UnaryInfC{$\Tforall_{x} \Tneg{\Tneg{P{x}}}$}
		\RightLabel{\Tunivelim}
		\UnaryInfC{$\Tneg{\Tneg{P{x}}}$}
			\AxiomC{}
			\UnaryInfC{$\Tneg{P{x}}$}
		\RightLabel{\Tarrowelim}
		\BinaryInfC{$\bot$}
	\RightLabel{\Texistelim}
	\BinaryInfC{$\bot$}
	\RightLabel{\Tarrowintro}
	\UnaryInfC{$\Tneg{\Texists_{x} \Tneg{P{x}}}$}
\RightLabel{\Tarrowelim}
\BinaryInfC{$\bot$}
\RightLabel{\Tarrowintro}
\UnaryInfC{$\Tneg{\Tneg{\Tforall_{x} P{x}}}$}
\RightLabel{\Tarrowintro}
\UnaryInfC{$\Tforall_{x} \Tneg{\Tneg{P{x}}} \Tarrow \Tneg{\Tneg{\Tforall_{x} P{x}}}$}
\end{deduction}
\end{proof}
\end{proposition}

\clearpage

\begin{proposition} \label{prop:dpalt}
$\DPi{Px}$ is equivalent to $\Texists_{y} \Tforall_{x} \left(P{y} \Tarrow P{x}\right)$
\begin{proof}
$(\implies)$
\begin{deduction}[nonfinal]
\AxiomC{$\Texists_{y} \left(P{y} \Tarrow \Tforall_{x} P{x}\right)$}
	\AxiomC{}
	\UnaryInfC{$P{y} \Tarrow \Tforall_{x} P{x}$}
		\AxiomC{}
		\UnaryInfC{$P{y}$}
	\RightLabel{\Tarrowelim}
	\BinaryInfC{$\Tforall_{x} P{x}$}
	\RightLabel{\Tunivelim}
	\UnaryInfC{$P{x}$}
	\RightLabel{\Tarrowintro}
	\UnaryInfC{$P{y} \Tarrow P{x}$}
	\RightLabel{\Tunivintro}
	\UnaryInfC{$\Tforall_{x} \left(P{y} \Tarrow P{x}\right)$}
	\RightLabel{\Texistintro}
	\UnaryInfC{$\Texists_{y} \Tforall_{x} \left(P{y} \Tarrow P{x}\right)$}
\RightLabel{\Texistelim}
\BinaryInfC{$\Texists_{y} \Tforall_{x} \left(P{y} \Tarrow P{x}\right)$}
\end{deduction}
\begin{deduction}
\AxiomC{$\Texists_{y} \Tforall_{x} \left(P{y} \Tarrow P{x}\right)$}
	\AxiomC{}
	\UnaryInfC{$\Tforall_{x} \left(P{y} \Tarrow P{x}\right)$}
	\RightLabel{\Tunivelim}
	\UnaryInfC{$P{y} \Tarrow P{x}$}
		\AxiomC{}
		\UnaryInfC{$P{y}$}
	\RightLabel{\Tarrowelim}
	\BinaryInfC{$P{x}$}
	\RightLabel{\Tunivintro}
	\UnaryInfC{$\Tforall_{x} P{x}$}
	\RightLabel{\Tarrowintro}
	\UnaryInfC{$P{y} \Tarrow \Tforall_{x} P{x}$}
	\RightLabel{\Texistintro}
	\UnaryInfC{$\Texists_{y} \left(P{y} \Tarrow \Tforall_{x} P{x}\right)$}
\RightLabel{\Texistelim}
\BinaryInfC{$\Texists_{y} \left(P{y} \Tarrow \Tforall_{x} P{x}\right)$}
\end{deduction}
\end{proof}
\end{proposition}

\clearpage

\begin{proposition} \label{prop:healt}
$\HEi{Px}$ is equivalent to $\Texists_{y} \Tforall_{x} \left(P{x} \Tarrow P{y}\right)$
\begin{proof}
$(\implies)$
\begin{deduction}[nonfinal]
\AxiomC{$\Texists_{y} \left(\Texists_{x} P{x} \Tarrow P{y}\right)$}
	\AxiomC{}
	\UnaryInfC{$\Texists_{x} P{x} \Tarrow P{y}$}
		\AxiomC{}
		\UnaryInfC{$P{x}$}
		\RightLabel{\Texistintro}
		\UnaryInfC{$\Texists_{x} P{x}$}
	\RightLabel{\Tarrowelim}
	\BinaryInfC{$P{y}$}
	\RightLabel{\Tarrowintro}
	\UnaryInfC{$P{x} \Tarrow P{y}$}
	\RightLabel{\Tunivintro}
	\UnaryInfC{$\Tforall_{x} \left(P{x} \Tarrow P{y}\right)$}
	\RightLabel{\Texistintro}
	\UnaryInfC{$\Texists_{y} \Tforall_{x} \left(P{x} \Tarrow P{y}\right)$}
\RightLabel{\Texistelim}
\BinaryInfC{$\Texists_{y} \Tforall_{x} \left(P{x} \Tarrow P{y}\right)$}
\end{deduction}
\begin{deduction}
\AxiomC{$\Texists_{y} \Tforall_{x} \left(P{x} \Tarrow P{y}\right)$}
	\AxiomC{}
	\UnaryInfC{$\Texists_{x} P{x}$}
		\AxiomC{}
		\UnaryInfC{$\Tforall_{x} \left(P{x} \Tarrow P{y}\right)$}
		\RightLabel{\Tunivelim}
		\UnaryInfC{$P{x} \Tarrow P{y}$}
			\AxiomC{}
			\UnaryInfC{$P{x}$}
		\RightLabel{\Tarrowelim}
		\BinaryInfC{$P{y}$}
	\RightLabel{\Texistelim}
	\BinaryInfC{$P{y}$}
	\RightLabel{\Tarrowintro}
	\UnaryInfC{$\Texists_{x} P{x} \Tarrow P{y}$}
	\RightLabel{\Texistintro}
	\UnaryInfC{$\Texists_{y} \left(\Texists_{x} P{x} \Tarrow P{y}\right)$}
\RightLabel{\Texistelim}
\BinaryInfC{$\Texists_{y} \left(\Texists_{x} P{x} \Tarrow P{y}\right)$}
\end{deduction}
\end{proof}
\end{proposition}

\clearpage

\begin{proposition} \label{prop:DNE,LEM,EFQ->DP}
$\DNEs, \LEMs, \EFQs \reduces \DPs$
\begin{proof}
First
\vspace{-\baselineskip}
\begin{prooftree}
\AxiomC{}
\RightLabel{DNE}
\UnaryInfC{$\Tneg{\Tneg{\Texists_{x} \Tneg{P{x}}}} \Tarrow \Texists_{x} \Tneg{P{x}}$}
	\AxiomC{$\Tneg{\Tforall_{x} P{x}}$}
		\AxiomC{}
		\RightLabel{DNE}
		\UnaryInfC{$\Tneg{\Tneg{P{x}}} \Tarrow P{x}$}
			\AxiomC{}
			\UnaryInfC{$\Tneg{\Texists_{x} \Tneg{P{x}}}$}
				\AxiomC{}
				\UnaryInfC{$\Tneg{P{x}}$}
				\RightLabel{\Texistintro}
				\UnaryInfC{$\Texists_{x} \Tneg{P{x}}$}
			\RightLabel{\Tarrowelim}
			\BinaryInfC{$\bot$}
			\RightLabel{\Tarrowintro}
			\UnaryInfC{$\Tneg{\Tneg{P{x}}}$}
		\RightLabel{\Tarrowelim}
		\BinaryInfC{$P{x}$}
		\RightLabel{\Tunivintro}
		\UnaryInfC{$\Tforall_{x} P{x}$}
	\RightLabel{\Tarrowelim}
	\BinaryInfC{$\bot$}
	\RightLabel{\Tarrowintro}
	\UnaryInfC{$\Tneg{\Tneg{\Texists_{x} \Tneg{P{x}}}}$}
\RightLabel{\Tarrowelim}
\BinaryInfC{$\Texists_{x} \Tneg{P{x}}$}
\end{prooftree}
\vspace{\baselineskip}Now,
\begin{deduction}
\AxiomC{}
\RightLabel{LEM}
\UnaryInfC{$\Tforall_{x} P{x} \Tor \Tneg{\Tforall_{x} P{x}}$}
	\AxiomC{}
	\UnaryInfC{$\Tforall_{x} P{x}$}
	\RightLabel{\Tarrowintro}
	\UnaryInfC{$P{y} \Tarrow \Tforall_{x} P{x}$}
	\RightLabel{\Texistintro}
	\UnaryInfC{$\Texists_{y} \left(P{y} \Tarrow \Tforall_{x} P{x}\right)$}
		\AxiomC{}
		\UnaryInfC{$\Tneg{\Tforall_{x} P{x}} \vdash_{DNE, LEM, EFQ} \Texists_{x} \Tneg{P{x}}$}
			\AxiomC{}
			\RightLabel{EFQ}
			\UnaryInfC{$\bot \Tarrow \Tforall_{x} P{x}$}
				\AxiomC{}
				\UnaryInfC{$\Tneg{P{x}}$}
					\AxiomC{}
					\UnaryInfC{$P{x}$}
				\RightLabel{\Tarrowelim}
				\BinaryInfC{$\bot$}
			\RightLabel{\Tarrowelim}
			\BinaryInfC{$\Tforall_{x} P{x}$}
			\RightLabel{\Tarrowintro}
			\UnaryInfC{$P{x} \Tarrow \Tforall_{x} P{x}$}
			\RightLabel{\Texistintro}
			\UnaryInfC{$\Texists_{y} \left(P{y} \Tarrow \Tforall_{x} P{x}\right)$}
		\RightLabel{\Texistelim}
		\BinaryInfC{$\Texists_{y} \left(P{y} \Tarrow \Tforall_{x} P{x}\right)$}
\RightLabel{\Tdisjelim}
\TrinaryInfC{$\Texists_{y} \left(P{y} \Tarrow \Tforall_{x} P{x}\right)$}
\end{deduction}
\end{proof}
\end{proposition}

\vspace{20pt}

\begin{proposition} \label{prop:LEM->WLEM}
$\LEMs \reduces \WLEMs$
\begin{proof}
\begin{deduction}
\AxiomC{}
\RightLabel{LEM}
\UnaryInfC{$\Tneg{A} \Tor \Tneg{\Tneg{A}}$}
\end{deduction}
\end{proof}
\end{proposition}
\clearpage

\begin{proposition} \label{prop:GMP->WGMP}
$\GMPs \reduces \WGMPs$
\begin{proof}
\begin{deduction}
\AxiomC{}
\UnaryInfC{$\Tneg{\Texists_{x} \Tneg{P{x}}}$}
	\AxiomC{}
	\RightLabel{GMP}
	\UnaryInfC{$\Tneg{\Tforall_{x} P{x}} \Tarrow \Texists_{x} \Tneg{P{x}}$}
		\AxiomC{}
		\UnaryInfC{$\Tneg{\Tforall_{x} P{x}}$}
	\RightLabel{\Tarrowelim}
	\BinaryInfC{$\Texists_{x} \Tneg{P{x}}$}
\RightLabel{\Tarrowelim}
\BinaryInfC{$\bot$}
\RightLabel{\Tarrowintro}
\UnaryInfC{$\Tneg{\Tneg{\Texists_{x} \Tneg{P{x}}}}$}
\RightLabel{\Tarrowintro}
\UnaryInfC{$\Tneg{\Tforall_{x} P{x}} \Tarrow \Tneg{\Tneg{\Texists_{x} \Tneg{P{x}}}}$}
\end{deduction}
\end{proof}
\end{proposition}

\begin{proposition} \label{prop:DGP->WLEM}
$\DGPs \reduces \WLEMs$
\begin{proof}
\begin{deduction}
\AxiomC{}
\RightLabel{DGP}
\UnaryInfC{$\left(A \Tarrow \Tneg{A}\right) \Tor \left(\Tneg{A} \Tarrow A\right)$}
	\AxiomC{}
	\UnaryInfC{$A \Tarrow \Tneg{A}$}
		\AxiomC{}
		\UnaryInfC{$A$}
	\RightLabel{\Tarrowelim}
	\BinaryInfC{$\Tneg{A}$}
		\AxiomC{}
		\UnaryInfC{$A$}
	\RightLabel{\Tarrowelim}
	\BinaryInfC{$\bot$}
	\RightLabel{\Tarrowintro}
	\UnaryInfC{$\Tneg{A}$}
	\RightLabel{\Tdisjintro}
	\UnaryInfC{$\Tneg{A} \Tor \Tneg{\Tneg{A}}$}
		\AxiomC{}
		\UnaryInfC{$\Tneg{A}$}
			\AxiomC{}
			\UnaryInfC{$\Tneg{A} \Tarrow A$}
				\AxiomC{}
				\UnaryInfC{$\Tneg{A}$}
			\RightLabel{\Tarrowelim}
			\BinaryInfC{$A$}
		\RightLabel{\Tarrowelim}
		\BinaryInfC{$\bot$}
		\RightLabel{\Tarrowintro}
		\UnaryInfC{$\Tneg{\Tneg{A}}$}
		\RightLabel{\Tdisjintro}
		\UnaryInfC{$\Tneg{A} \Tor \Tneg{\Tneg{A}}$}
\RightLabel{\Tdisjelim}
\TrinaryInfC{$\Tneg{A} \Tor \Tneg{\Tneg{A}}$}
\end{deduction}
\end{proof}
\end{proposition}

\begin{proposition} \label{prop:GLPOA->LEM}
$\GLPOAs \reduces \LEMs$
\begin{proof}
\begin{deduction}
\AxiomC{}
\RightLabel{GLPOA}
\UnaryInfC{$\Tforall_{x} A \Tor \Texists_{x} \Tneg{A}$}
	\AxiomC{}
	\UnaryInfC{$\Tforall_{x} A$}
	\RightLabel{\Tunivelim}
	\UnaryInfC{$A$}
	\RightLabel{\Tdisjintro}
	\UnaryInfC{$A \Tor \Tneg{A}$}
		\AxiomC{}
		\UnaryInfC{$\Texists_{x} \Tneg{A}$}
			\AxiomC{}
			\UnaryInfC{$\Tneg{A}$}
		\RightLabel{\Texistelim}
		\BinaryInfC{$\Tneg{A}$}
		\RightLabel{\Tdisjintro}
		\UnaryInfC{$A \Tor \Tneg{A}$}
\RightLabel{\Tdisjelim}
\TrinaryInfC{$A \Tor \Tneg{A}$}
\end{deduction}
\end{proof}
\end{proposition}

\begin{proposition} \label{prop:GLPOA->GMP}
$\GLPOAs \reduces \GMPs$
\begin{proof}
\begin{deduction}
\AxiomC{}
\RightLabel{GLPOA}
\UnaryInfC{$\Tforall_{x} P{x} \Tor \Texists_{x} \Tneg{P{x}}$}
	\AxiomC{}
	\UnaryInfC{$\Tneg{\Tforall_{x} P{x}}$}
		\AxiomC{}
		\UnaryInfC{$\Tforall_{x} P{x}$}
	\RightLabel{\Tarrowelim}
	\BinaryInfC{$\bot$}
	\RightLabel{\Tarrowintro}
	\UnaryInfC{$\Tneg{P{x}}$}
	\RightLabel{\Texistintro}
	\UnaryInfC{$\Texists_{x} \Tneg{P{x}}$}
		\AxiomC{}
		\UnaryInfC{$\Texists_{x} \Tneg{P{x}}$}
\RightLabel{\Tdisjelim}
\TrinaryInfC{$\Texists_{x} \Tneg{P{x}}$}
\RightLabel{\Tarrowintro}
\UnaryInfC{$\Tneg{\Tforall_{x} P{x}} \Tarrow \Texists_{x} \Tneg{P{x}}$}
\end{deduction}
\end{proof}
\end{proposition}

\clearpage

\begin{proposition} \label{prop:DP->CD}
$\DPs \reduces \CDs$
\begin{proof}
\begin{deduction}
\AxiomC{}
\RightLabel{DP}
\UnaryInfC{$\Texists_{y} \left(P{y} \Tarrow \Tforall_{x} P{x}\right)$}
	\AxiomC{}
	\UnaryInfC{$\Tforall_{x} \left(P{x} \Tor \Texists_{x} A\right)$}
	\RightLabel{\Tunivelim}
	\UnaryInfC{$P{y} \Tor \Texists_{x} A$}
		\AxiomC{}
		\UnaryInfC{$P{y} \Tarrow \Tforall_{x} P{x}$}
			\AxiomC{}
			\UnaryInfC{$P{y}$}
		\RightLabel{\Tarrowelim}
		\BinaryInfC{$\Tforall_{x} P{x}$}
		\RightLabel{\Tdisjintro}
		\UnaryInfC{$\Tforall_{x} P{x} \Tor \Texists_{x} A$}
			\AxiomC{}
			\UnaryInfC{$\Texists_{x} A$}
			\RightLabel{\Tdisjintro}
			\UnaryInfC{$\Tforall_{x} P{x} \Tor \Texists_{x} A$}
	\RightLabel{\Tdisjelim}
	\TrinaryInfC{$\Tforall_{x} P{x} \Tor \Texists_{x} A$}
\RightLabel{\Texistelim}
\BinaryInfC{$\Tforall_{x} P{x} \Tor \Texists_{x} A$}
\RightLabel{\Tarrowintro}
\UnaryInfC{$\Tforall_{x} \left(P{x} \Tor \Texists_{x} A\right) \Tarrow \left(\Tforall_{x} P{x} \Tor \Texists_{x} A\right)$}
\end{deduction}
\end{proof}
\end{proposition}

\begin{proposition} \label{prop:DP->GMP}
$\DPs \reduces \GMPs$
\begin{proof}
\begin{deduction}
\AxiomC{}
\RightLabel{DP}
\UnaryInfC{$\Texists_{y} \left(P{y} \Tarrow \Tforall_{x} P{x}\right)$}
	\AxiomC{}
	\UnaryInfC{$\Tneg{\Tforall_{x} P{x}}$}
		\AxiomC{}
		\UnaryInfC{$P{y} \Tarrow \Tforall_{x} P{x}$}
			\AxiomC{}
			\UnaryInfC{$P{y}$}
		\RightLabel{\Tarrowelim}
		\BinaryInfC{$\Tforall_{x} P{x}$}
	\RightLabel{\Tarrowelim}
	\BinaryInfC{$\bot$}
	\RightLabel{\Tarrowintro}
	\UnaryInfC{$\Tneg{P{y}}$}
	\RightLabel{\Texistintro}
	\UnaryInfC{$\Texists_{x} \Tneg{P{x}}$}
\RightLabel{\Texistelim}
\BinaryInfC{$\Texists_{x} \Tneg{P{x}}$}
\RightLabel{\Tarrowintro}
\UnaryInfC{$\Tneg{\Tforall_{x} P{x}} \Tarrow \Texists_{x} \Tneg{P{x}}$}
\end{deduction}
\end{proof}
\end{proposition}

\begin{proposition} \label{prop:HE->DNSE}
$\HEs \reduces \DNSEs$
\begin{proof}
\begin{deduction}
\AxiomC{}
\RightLabel{HE}
\UnaryInfC{$\Texists_{y} \left(\Texists_{x} P{x} \Tarrow P{y}\right)$}
	\AxiomC{}
	\UnaryInfC{$\Tneg{\Tneg{\Texists_{x} P{x}}}$}
		\AxiomC{}
		\UnaryInfC{$\Tneg{P{y}}$}
			\AxiomC{}
			\UnaryInfC{$\Texists_{x} P{x} \Tarrow P{y}$}
				\AxiomC{}
				\UnaryInfC{$\Texists_{x} P{x}$}
			\RightLabel{\Tarrowelim}
			\BinaryInfC{$P{y}$}
		\RightLabel{\Tarrowelim}
		\BinaryInfC{$\bot$}
		\RightLabel{\Tarrowintro}
		\UnaryInfC{$\Tneg{\Texists_{x} P{x}}$}
	\RightLabel{\Tarrowelim}
	\BinaryInfC{$\bot$}
	\RightLabel{\Tarrowintro}
	\UnaryInfC{$\Tneg{\Tneg{P{y}}}$}
	\RightLabel{\Texistintro}
	\UnaryInfC{$\Texists_{x} \Tneg{\Tneg{P{x}}}$}
\RightLabel{\Texistelim}
\BinaryInfC{$\Texists_{x} \Tneg{\Tneg{P{x}}}$}
\RightLabel{\Tarrowintro}
\UnaryInfC{$\Tneg{\Tneg{\Texists_{x} P{x}}} \Tarrow \Texists_{x} \Tneg{\Tneg{P{x}}}$}
\end{deduction}
\end{proof}
\end{proposition}

\clearpage

\begin{proposition} \label{prop:GLPO->DNSE}
$\GLPOs \reduces \DNSEs$
\begin{proof}
\begin{deduction}
\AxiomC{}
\RightLabel{GLPO}
\UnaryInfC{$\Tforall_{x} \Tneg{P{x}} \Tor \Texists_{x} P{x}$}
	\AxiomC{}
	\UnaryInfC{$\Tneg{\Tneg{\Texists_{x} P{x}}}$}
		\AxiomC{}
		\UnaryInfC{$\Texists_{x} P{x}$}
			\AxiomC{}
			\UnaryInfC{$\Tforall_{x} \Tneg{P{x}}$}
			\RightLabel{\Tunivelim}
			\UnaryInfC{$\Tneg{P{x}}$}
				\AxiomC{}
				\UnaryInfC{$P{x}$}
			\RightLabel{\Tarrowelim}
			\BinaryInfC{$\bot$}
		\RightLabel{\Texistelim}
		\BinaryInfC{$\bot$}
		\RightLabel{\Tarrowintro}
		\UnaryInfC{$\Tneg{\Texists_{x} P{x}}$}
	\RightLabel{\Tarrowelim}
	\BinaryInfC{$\bot$}
	\RightLabel{\Tarrowintro}
	\UnaryInfC{$\Tneg{\Tneg{P{x}}}$}
	\RightLabel{\Texistintro}
	\UnaryInfC{$\Texists_{x} \Tneg{\Tneg{P{x}}}$}
		\AxiomC{}
		\UnaryInfC{$\Texists_{x} P{x}$}
			\AxiomC{}
			\UnaryInfC{$\Tneg{P{x}}$}
				\AxiomC{}
				\UnaryInfC{$P{x}$}
			\RightLabel{\Tarrowelim}
			\BinaryInfC{$\bot$}
			\RightLabel{\Tarrowintro}
			\UnaryInfC{$\Tneg{\Tneg{P{x}}}$}
			\RightLabel{\Texistintro}
			\UnaryInfC{$\Texists_{x} \Tneg{\Tneg{P{x}}}$}
		\RightLabel{\Texistelim}
		\BinaryInfC{$\Texists_{x} \Tneg{\Tneg{P{x}}}$}
\RightLabel{\Tdisjelim}
\TrinaryInfC{$\Texists_{x} \Tneg{\Tneg{P{x}}}$}
\RightLabel{\Tarrowintro}
\UnaryInfC{$\Tneg{\Tneg{\Texists_{x} P{x}}} \Tarrow \Texists_{x} \Tneg{\Tneg{P{x}}}$}
\end{deduction}
\end{proof}
\end{proposition}

\begin{proposition} \label{prop:GMP->DNSE}
$\GMPs \reduces \DNSEs$
\begin{proof}
\begin{deduction}
\AxiomC{}
\RightLabel{GMP}
\UnaryInfC{$\Tneg{\Tforall_{x} \Tneg{P{x}}} \Tarrow \Texists_{x} \Tneg{\Tneg{P{x}}}$}
	\AxiomC{}
	\UnaryInfC{$\Tneg{\Tneg{\Texists_{x} P{x}}}$}
		\AxiomC{}
		\UnaryInfC{$\Texists_{x} P{x}$}
			\AxiomC{}
			\UnaryInfC{$\Tforall_{x} \Tneg{P{x}}$}
			\RightLabel{\Tunivelim}
			\UnaryInfC{$\Tneg{P{x}}$}
				\AxiomC{}
				\UnaryInfC{$P{x}$}
			\RightLabel{\Tarrowelim}
			\BinaryInfC{$\bot$}
		\RightLabel{\Texistelim}
		\BinaryInfC{$\bot$}
		\RightLabel{\Tarrowintro}
		\UnaryInfC{$\Tneg{\Texists_{x} P{x}}$}
	\RightLabel{\Tarrowelim}
	\BinaryInfC{$\bot$}
	\RightLabel{\Tarrowintro}
	\UnaryInfC{$\Tneg{\Tforall_{x} \Tneg{P{x}}}$}
\RightLabel{\Tarrowelim}
\BinaryInfC{$\Texists_{x} \Tneg{\Tneg{P{x}}}$}
\RightLabel{\Tarrowintro}
\UnaryInfC{$\Tneg{\Tneg{\Texists_{x} P{x}}} \Tarrow \Texists_{x} \Tneg{\Tneg{P{x}}}$}
\end{deduction}
\end{proof}
\end{proposition}

\begin{proposition} \label{prop:GLPOA->WGMP}
$\GLPOAs \reduces \WGMPs$
\begin{proof}
\begin{deduction}
\AxiomC{}
\RightLabel{GLPOA}
\UnaryInfC{$\Tforall_{x} P{x} \Tor \Texists_{x} \Tneg{P{x}}$}
	\AxiomC{}
	\UnaryInfC{$\Tneg{\Tforall_{x} P{x}}$}
		\AxiomC{}
		\UnaryInfC{$\Tforall_{x} P{x}$}
	\RightLabel{\Tarrowelim}
	\BinaryInfC{$\bot$}
	\RightLabel{\Tarrowintro}
	\UnaryInfC{$\Tneg{\Tneg{\Texists_{x} \Tneg{P{x}}}}$}
	\RightLabel{\Tarrowintro}
	\UnaryInfC{$\Tneg{\Tforall_{x} P{x}} \Tarrow \Tneg{\Tneg{\Texists_{x} \Tneg{P{x}}}}$}
		\AxiomC{}
		\UnaryInfC{$\Tneg{\Texists_{x} \Tneg{P{x}}}$}
			\AxiomC{}
			\UnaryInfC{$\Texists_{x} \Tneg{P{x}}$}
		\RightLabel{\Tarrowelim}
		\BinaryInfC{$\bot$}
		\RightLabel{\Tarrowintro}
		\UnaryInfC{$\Tneg{\Tneg{\Texists_{x} \Tneg{P{x}}}}$}
		\RightLabel{\Tarrowintro}
		\UnaryInfC{$\Tneg{\Tforall_{x} P{x}} \Tarrow \Tneg{\Tneg{\Texists_{x} \Tneg{P{x}}}}$}
\RightLabel{\Tdisjelim}
\TrinaryInfC{$\Tneg{\Tforall_{x} P{x}} \Tarrow \Tneg{\Tneg{\Texists_{x} \Tneg{P{x}}}}$}
\end{deduction}
\end{proof}
\end{proposition}

\clearpage

\begin{proposition} \label{prop:DP,EFQ,DZ,DO,DX->DGP}
$\DPs, \EFQs, \TT/ \reduces \DGPs$
\begin{proof}
Where $\Phi = \left(\left(D{y} \Tarrow A\right) \Tand \left(\Tneg{D{y}} \Tarrow
B\right)\right) \Tarrow \Tforall_{x} \left(\left(D{x} \Tarrow A\right) \Tand
\left(\Tneg{D{x}} \Tarrow B\right)\right)$, \\
Lemma 1:
\vspace{-\baselineskip}
\begin{prooftree}
\AxiomC{$\Phi$}
	\AxiomC{}
	\UnaryInfC{$A$}
	\RightLabel{\Tarrowintro}
	\UnaryInfC{$D{y} \Tarrow A$}
		\AxiomC{}
		\RightLabel{EFQ}
		\UnaryInfC{$\bot \Tarrow B$}
			\AxiomC{}
			\UnaryInfC{$\Tneg{D{y}}$}
				\AxiomC{$D{y}$}
			\RightLabel{\Tarrowelim}
			\BinaryInfC{$\bot$}
		\RightLabel{\Tarrowelim}
		\BinaryInfC{$B$}
		\RightLabel{\Tarrowintro}
		\UnaryInfC{$\Tneg{D{y}} \Tarrow B$}
	\RightLabel{\Tconjintro}
	\BinaryInfC{$\left(D{y} \Tarrow A\right) \Tand \left(\Tneg{D{y}} \Tarrow B\right)$}
\RightLabel{\Tarrowelim}
\BinaryInfC{$\Tforall_{x} \left(\left(D{x} \Tarrow A\right) \Tand \left(\Tneg{D{x}} \Tarrow B\right)\right)$}
\RightLabel{\Tunivelim}
\UnaryInfC{$\left(D{1} \Tarrow A\right) \Tand \left(\Tneg{D{1}} \Tarrow B\right)$}
	\AxiomC{}
	\UnaryInfC{$\Tneg{D{1}} \Tarrow B$}
		\AxiomC{}
		\RightLabel{DO}
		\UnaryInfC{$\Tneg{D{1}}$}
	\RightLabel{\Tarrowelim}
	\BinaryInfC{$B$}
\RightLabel{\Tconjelim}
\BinaryInfC{$B$}
\RightLabel{\Tarrowintro}
\UnaryInfC{$A \Tarrow B$}
\RightLabel{\Tdisjintro}
\UnaryInfC{$\left(A \Tarrow B\right) \Tor \left(B \Tarrow A\right)$}
\end{prooftree}
\vspace{\baselineskip}Lemma 2:
\vspace{-\baselineskip}
\begin{prooftree}
\AxiomC{$\Phi$}
	\AxiomC{}
	\RightLabel{EFQ}
	\UnaryInfC{$\bot \Tarrow A$}
		\AxiomC{$\Tneg{D{y}}$}
			\AxiomC{}
			\UnaryInfC{$D{y}$}
		\RightLabel{\Tarrowelim}
		\BinaryInfC{$\bot$}
	\RightLabel{\Tarrowelim}
	\BinaryInfC{$A$}
	\RightLabel{\Tarrowintro}
	\UnaryInfC{$D{y} \Tarrow A$}
		\AxiomC{}
		\UnaryInfC{$B$}
		\RightLabel{\Tarrowintro}
		\UnaryInfC{$\Tneg{D{y}} \Tarrow B$}
	\RightLabel{\Tconjintro}
	\BinaryInfC{$\left(D{y} \Tarrow A\right) \Tand \left(\Tneg{D{y}} \Tarrow B\right)$}
\RightLabel{\Tarrowelim}
\BinaryInfC{$\Tforall_{x} \left(\left(D{x} \Tarrow A\right) \Tand \left(\Tneg{D{x}} \Tarrow B\right)\right)$}
\RightLabel{\Tunivelim}
\UnaryInfC{$\left(D{0} \Tarrow A\right) \Tand \left(\Tneg{D{0}} \Tarrow B\right)$}
	\AxiomC{}
	\UnaryInfC{$D{0} \Tarrow A$}
		\AxiomC{}
		\RightLabel{DZ}
		\UnaryInfC{$D{0}$}
	\RightLabel{\Tarrowelim}
	\BinaryInfC{$A$}
\RightLabel{\Tconjelim}
\BinaryInfC{$A$}
\RightLabel{\Tarrowintro}
\UnaryInfC{$B \Tarrow A$}
\RightLabel{\Tdisjintro}
\UnaryInfC{$\left(A \Tarrow B\right) \Tor \left(B \Tarrow A\right)$}
\end{prooftree}
\vspace{\baselineskip}Now,
\begin{deduction}
\AxiomC{}
\RightLabel{DP}
\UnaryInfC{$\Texists_{y} \Phi$}
	\AxiomC{}
	\RightLabel{DX}
	\UnaryInfC{$\Tforall_{x} \left(D{x} \Tor \Tneg{D{x}}\right)$}
	\RightLabel{\Tunivelim}
	\UnaryInfC{$D{y} \Tor \Tneg{D{y}}$}
		\AxiomC{}
		\UnaryInfC{Lemma 1}
			\AxiomC{}
			\UnaryInfC{Lemma 2}
	\RightLabel{\Tdisjelim}
	\TrinaryInfC{$\left(A \Tarrow B\right) \Tor \left(B \Tarrow A\right)$}
\RightLabel{\Texistelim}
\BinaryInfC{$\left(A \Tarrow B\right) \Tor \left(B \Tarrow A\right)$}
\end{deduction}
\end{proof}
\end{proposition}

\clearpage

\begin{proposition} \label{prop:DP,DZ,DO,DX->WLEM}
$\DPs, \TT/ \reduces \WLEMs$
\begin{proof}
Where $\Phi = \left(\left(D{y} \Tarrow \Tneg{\Tneg{A}}\right) \Tand
\left(\Tneg{D{y}} \Tarrow \Tneg{A}\right)\right) \Tarrow \Tforall_{x}
\left(\left(D{x} \Tarrow \Tneg{\Tneg{A}}\right) \Tand \left(\Tneg{D{x}} \Tarrow
\Tneg{A}\right)\right)$,\\
Lemma 1:
\vspace{-\baselineskip}
\begin{prooftree}
\AxiomC{$\Phi$}
	\AxiomC{}
	\UnaryInfC{$\Tneg{A}$}
		\AxiomC{}
		\UnaryInfC{$A$}
	\RightLabel{\Tarrowelim}
	\BinaryInfC{$\bot$}
	\RightLabel{\Tarrowintro}
	\UnaryInfC{$\Tneg{\Tneg{A}}$}
	\RightLabel{\Tarrowintro}
	\UnaryInfC{$D{y} \Tarrow \Tneg{\Tneg{A}}$}
		\AxiomC{}
		\UnaryInfC{$\Tneg{D{y}}$}
			\AxiomC{$D{y}$}
		\RightLabel{\Tarrowelim}
		\BinaryInfC{$\bot$}
		\RightLabel{\Tarrowintro}
		\UnaryInfC{$\Tneg{A}$}
		\RightLabel{\Tarrowintro}
		\UnaryInfC{$\Tneg{D{y}} \Tarrow \Tneg{A}$}
	\RightLabel{\Tconjintro}
	\BinaryInfC{$\left(D{y} \Tarrow \Tneg{\Tneg{A}}\right) \Tand \left(\Tneg{D{y}} \Tarrow \Tneg{A}\right)$}
\RightLabel{\Tarrowelim}
\BinaryInfC{$\Tforall_{x} \left(\left(D{x} \Tarrow \Tneg{\Tneg{A}}\right) \Tand \left(\Tneg{D{x}} \Tarrow \Tneg{A}\right)\right)$}
\RightLabel{\Tunivelim}
\UnaryInfC{$\left(D{1} \Tarrow \Tneg{\Tneg{A}}\right) \Tand \left(\Tneg{D{1}} \Tarrow \Tneg{A}\right)$}
	\AxiomC{}
	\UnaryInfC{$\Tneg{D{1}} \Tarrow \Tneg{A}$}
		\AxiomC{}
		\RightLabel{DO}
		\UnaryInfC{$\Tneg{D{1}}$}
	\RightLabel{\Tarrowelim}
	\BinaryInfC{$\Tneg{A}$}
		\AxiomC{}
		\UnaryInfC{$A$}
	\RightLabel{\Tarrowelim}
	\BinaryInfC{$\bot$}
\RightLabel{\Tconjelim}
\BinaryInfC{$\bot$}
\RightLabel{\Tarrowintro}
\UnaryInfC{$\Tneg{A}$}
\RightLabel{\Tdisjintro}
\UnaryInfC{$\Tneg{A} \Tor \Tneg{\Tneg{A}}$}
\end{prooftree}
\vspace{\baselineskip}Lemma 2:
\vspace{-\baselineskip}
\begin{prooftree}
\AxiomC{$\Phi$}
	\AxiomC{$\Tneg{D{y}}$}
		\AxiomC{}
		\UnaryInfC{$D{y}$}
	\RightLabel{\Tarrowelim}
	\BinaryInfC{$\bot$}
	\RightLabel{\Tarrowintro}
	\UnaryInfC{$\Tneg{\Tneg{A}}$}
	\RightLabel{\Tarrowintro}
	\UnaryInfC{$D{y} \Tarrow \Tneg{\Tneg{A}}$}
		\AxiomC{}
		\UnaryInfC{$\Tneg{A}$}
		\RightLabel{\Tarrowintro}
		\UnaryInfC{$\Tneg{D{y}} \Tarrow \Tneg{A}$}
	\RightLabel{\Tconjintro}
	\BinaryInfC{$\left(D{y} \Tarrow \Tneg{\Tneg{A}}\right) \Tand \left(\Tneg{D{y}} \Tarrow \Tneg{A}\right)$}
\RightLabel{\Tarrowelim}
\BinaryInfC{$\Tforall_{x} \left(\left(D{x} \Tarrow \Tneg{\Tneg{A}}\right) \Tand \left(\Tneg{D{x}} \Tarrow \Tneg{A}\right)\right)$}
\RightLabel{\Tunivelim}
\UnaryInfC{$\left(D{0} \Tarrow \Tneg{\Tneg{A}}\right) \Tand \left(\Tneg{D{0}} \Tarrow \Tneg{A}\right)$}
	\AxiomC{}
	\UnaryInfC{$D{0} \Tarrow \Tneg{\Tneg{A}}$}
		\AxiomC{}
		\RightLabel{DZ}
		\UnaryInfC{$D{0}$}
	\RightLabel{\Tarrowelim}
	\BinaryInfC{$\Tneg{\Tneg{A}}$}
		\AxiomC{}
		\UnaryInfC{$\Tneg{A}$}
	\RightLabel{\Tarrowelim}
	\BinaryInfC{$\bot$}
\RightLabel{\Tconjelim}
\BinaryInfC{$\bot$}
\RightLabel{\Tarrowintro}
\UnaryInfC{$\Tneg{\Tneg{A}}$}
\RightLabel{\Tdisjintro}
\UnaryInfC{$\Tneg{A} \Tor \Tneg{\Tneg{A}}$}
\end{prooftree}
\vspace{\baselineskip}Now,
\begin{deduction}
\AxiomC{}
\RightLabel{DP}
\UnaryInfC{$\Texists_{y} \Phi$}
	\AxiomC{}
	\RightLabel{DX}
	\UnaryInfC{$\Tforall_{x} \left(D{x} \Tor \Tneg{D{x}}\right)$}
	\RightLabel{\Tunivelim}
	\UnaryInfC{$D{y} \Tor \Tneg{D{y}}$}
		\AxiomC{}
		\UnaryInfC{Lemma 1}
			\AxiomC{}
			\UnaryInfC{Lemma 2}
	\RightLabel{\Tdisjelim}
	\TrinaryInfC{$\Tneg{A} \Tor \Tneg{\Tneg{A}}$}
\RightLabel{\Texistelim}
\BinaryInfC{$\Tneg{A} \Tor \Tneg{\Tneg{A}}$}
\end{deduction}
\end{proof}
\end{proposition}

\clearpage

\begin{proposition} \label{prop:HE,EFQ,DZ,DO,DX->DGP}
$\HEs, \EFQs, \TT/ \reduces \DGPs$
\begin{proof}
Where $\Phi = \Texists_{x} \left(\left(D{x} \Tarrow A\right) \Tand
\left(\Tneg{D{x}} \Tarrow B\right)\right) \Tarrow \left(\left(D{y} \Tarrow
A\right) \Tand \left(\Tneg{D{y}} \Tarrow B\right)\right)$,\\
Lemma 1:
\vspace{-\baselineskip}
\begin{prooftree}
\AxiomC{$\Phi$}
	\AxiomC{}
	\RightLabel{EFQ}
	\UnaryInfC{$\bot \Tarrow A$}
		\AxiomC{}
		\RightLabel{DO}
		\UnaryInfC{$\Tneg{D{1}}$}
			\AxiomC{}
			\UnaryInfC{$D{1}$}
		\RightLabel{\Tarrowelim}
		\BinaryInfC{$\bot$}
	\RightLabel{\Tarrowelim}
	\BinaryInfC{$A$}
	\RightLabel{\Tarrowintro}
	\UnaryInfC{$D{1} \Tarrow A$}
		\AxiomC{}
		\UnaryInfC{$B$}
		\RightLabel{\Tarrowintro}
		\UnaryInfC{$\Tneg{D{1}} \Tarrow B$}
	\RightLabel{\Tconjintro}
	\BinaryInfC{$\left(D{1} \Tarrow A\right) \Tand \left(\Tneg{D{1}} \Tarrow B\right)$}
	\RightLabel{\Texistintro}
	\UnaryInfC{$\Texists_{x} \left(\left(D{x} \Tarrow A\right) \Tand \left(\Tneg{D{x}} \Tarrow B\right)\right)$}
\RightLabel{\Tarrowelim}
\BinaryInfC{$\left(D{y} \Tarrow A\right) \Tand \left(\Tneg{D{y}} \Tarrow B\right)$}
	\AxiomC{}
	\UnaryInfC{$D{y} \Tarrow A$}
		\AxiomC{$D{y}$}
	\RightLabel{\Tarrowelim}
	\BinaryInfC{$A$}
\RightLabel{\Tconjelim}
\BinaryInfC{$A$}
\RightLabel{\Tarrowintro}
\UnaryInfC{$B \Tarrow A$}
\RightLabel{\Tdisjintro}
\UnaryInfC{$\left(A \Tarrow B\right) \Tor \left(B \Tarrow A\right)$}
\end{prooftree}
\vspace{\baselineskip}Lemma 2:
\vspace{-\baselineskip}
\begin{prooftree}
\AxiomC{$\Phi$}
	\AxiomC{}
	\UnaryInfC{$A$}
	\RightLabel{\Tarrowintro}
	\UnaryInfC{$D{0} \Tarrow A$}
		\AxiomC{}
		\RightLabel{EFQ}
		\UnaryInfC{$\bot \Tarrow B$}
			\AxiomC{}
			\UnaryInfC{$\Tneg{D{0}}$}
				\AxiomC{}
				\RightLabel{DZ}
				\UnaryInfC{$D{0}$}
			\RightLabel{\Tarrowelim}
			\BinaryInfC{$\bot$}
		\RightLabel{\Tarrowelim}
		\BinaryInfC{$B$}
		\RightLabel{\Tarrowintro}
		\UnaryInfC{$\Tneg{D{0}} \Tarrow B$}
	\RightLabel{\Tconjintro}
	\BinaryInfC{$\left(D{0} \Tarrow A\right) \Tand \left(\Tneg{D{0}} \Tarrow B\right)$}
	\RightLabel{\Texistintro}
	\UnaryInfC{$\Texists_{x} \left(\left(D{x} \Tarrow A\right) \Tand \left(\Tneg{D{x}} \Tarrow B\right)\right)$}
\RightLabel{\Tarrowelim}
\BinaryInfC{$\left(D{y} \Tarrow A\right) \Tand \left(\Tneg{D{y}} \Tarrow B\right)$}
	\AxiomC{}
	\UnaryInfC{$\Tneg{D{y}} \Tarrow B$}
		\AxiomC{$\Tneg{D{y}}$}
	\RightLabel{\Tarrowelim}
	\BinaryInfC{$B$}
\RightLabel{\Tconjelim}
\BinaryInfC{$B$}
\RightLabel{\Tarrowintro}
\UnaryInfC{$A \Tarrow B$}
\RightLabel{\Tdisjintro}
\UnaryInfC{$\left(A \Tarrow B\right) \Tor \left(B \Tarrow A\right)$}
\end{prooftree}
\vspace{\baselineskip}Now,
\begin{deduction}
\AxiomC{}
\RightLabel{HE}
\UnaryInfC{$\Texists_{y} \Phi$}
	\AxiomC{}
	\RightLabel{DX}
	\UnaryInfC{$\Tforall_{x} \left(D{x} \Tor \Tneg{D{x}}\right)$}
	\RightLabel{\Tunivelim}
	\UnaryInfC{$D{y} \Tor \Tneg{D{y}}$}
		\AxiomC{}
		\UnaryInfC{Lemma 1}
			\AxiomC{}
			\UnaryInfC{Lemma 2}
	\RightLabel{\Tdisjelim}
	\TrinaryInfC{$\left(A \Tarrow B\right) \Tor \left(B \Tarrow A\right)$}
\RightLabel{\Texistelim}
\BinaryInfC{$\left(A \Tarrow B\right) \Tor \left(B \Tarrow A\right)$}
\end{deduction}
\end{proof}
\end{proposition}

\clearpage

\begin{proposition} \label{prop:HE,DZ,DO,DX->WLEM}
$\HEs, \TT/ \reduces \WLEMs$
\begin{proof}
Where $\Phi = \Texists_{x} \left(\left(D{x} \Tarrow \Tneg{\Tneg{A}}\right) \Tand
\left(\Tneg{D{x}} \Tarrow \Tneg{A}\right)\right) \Tarrow \left(\left(D{y}
\Tarrow \Tneg{\Tneg{A}}\right) \Tand \left(\Tneg{D{y}} \Tarrow
\Tneg{A}\right)\right)$,\\
Lemma 1:
\vspace{-\baselineskip}
\begin{prooftree}
\AxiomC{$\Phi$}
	\AxiomC{}
	\RightLabel{DO}
	\UnaryInfC{$\Tneg{D{1}}$}
		\AxiomC{}
		\UnaryInfC{$D{1}$}
	\RightLabel{\Tarrowelim}
	\BinaryInfC{$\bot$}
	\RightLabel{\Tarrowintro}
	\UnaryInfC{$\Tneg{\Tneg{A}}$}
	\RightLabel{\Tarrowintro}
	\UnaryInfC{$D{1} \Tarrow \Tneg{\Tneg{A}}$}
		\AxiomC{}
		\UnaryInfC{$\Tneg{A}$}
		\RightLabel{\Tarrowintro}
		\UnaryInfC{$\Tneg{D{1}} \Tarrow \Tneg{A}$}
	\RightLabel{\Tconjintro}
	\BinaryInfC{$\left(D{1} \Tarrow \Tneg{\Tneg{A}}\right) \Tand \left(\Tneg{D{1}} \Tarrow \Tneg{A}\right)$}
	\RightLabel{\Texistintro}
	\UnaryInfC{$\Texists_{x} \left(\left(D{x} \Tarrow \Tneg{\Tneg{A}}\right) \Tand \left(\Tneg{D{x}} \Tarrow \Tneg{A}\right)\right)$}
\RightLabel{\Tarrowelim}
\BinaryInfC{$\left(D{y} \Tarrow \Tneg{\Tneg{A}}\right) \Tand \left(\Tneg{D{y}} \Tarrow \Tneg{A}\right)$}
	\AxiomC{}
	\UnaryInfC{$D{y} \Tarrow \Tneg{\Tneg{A}}$}
		\AxiomC{$D{y}$}
	\RightLabel{\Tarrowelim}
	\BinaryInfC{$\Tneg{\Tneg{A}}$}
		\AxiomC{}
		\UnaryInfC{$\Tneg{A}$}
	\RightLabel{\Tarrowelim}
	\BinaryInfC{$\bot$}
\RightLabel{\Tconjelim}
\BinaryInfC{$\bot$}
\RightLabel{\Tarrowintro}
\UnaryInfC{$\Tneg{\Tneg{A}}$}
\RightLabel{\Tdisjintro}
\UnaryInfC{$\Tneg{A} \Tor \Tneg{\Tneg{A}}$}
\end{prooftree}
Lemma 2:
\vspace{-\baselineskip}
\begin{prooftree}
\AxiomC{$\Phi$}
	\AxiomC{}
	\UnaryInfC{$\Tneg{A}$}
		\AxiomC{}
		\UnaryInfC{$A$}
	\RightLabel{\Tarrowelim}
	\BinaryInfC{$\bot$}
	\RightLabel{\Tarrowintro}
	\UnaryInfC{$\Tneg{\Tneg{A}}$}
	\RightLabel{\Tarrowintro}
	\UnaryInfC{$D{0} \Tarrow \Tneg{\Tneg{A}}$}
		\AxiomC{}
		\UnaryInfC{$\Tneg{D{0}}$}
			\AxiomC{}
			\RightLabel{DZ}
			\UnaryInfC{$D{0}$}
		\RightLabel{\Tarrowelim}
		\BinaryInfC{$\bot$}
		\RightLabel{\Tarrowintro}
		\UnaryInfC{$\Tneg{A}$}
		\RightLabel{\Tarrowintro}
		\UnaryInfC{$\Tneg{D{0}} \Tarrow \Tneg{A}$}
	\RightLabel{\Tconjintro}
	\BinaryInfC{$\left(D{0} \Tarrow \Tneg{\Tneg{A}}\right) \Tand \left(\Tneg{D{0}} \Tarrow \Tneg{A}\right)$}
	\RightLabel{\Texistintro}
	\UnaryInfC{$\Texists_{x} \left(\left(D{x} \Tarrow \Tneg{\Tneg{A}}\right) \Tand \left(\Tneg{D{x}} \Tarrow \Tneg{A}\right)\right)$}
\RightLabel{\Tarrowelim}
\BinaryInfC{$\left(D{y} \Tarrow \Tneg{\Tneg{A}}\right) \Tand \left(\Tneg{D{y}} \Tarrow \Tneg{A}\right)$}
	\AxiomC{}
	\UnaryInfC{$\Tneg{D{y}} \Tarrow \Tneg{A}$}
		\AxiomC{$\Tneg{D{y}}$}
	\RightLabel{\Tarrowelim}
	\BinaryInfC{$\Tneg{A}$}
		\AxiomC{}
		\UnaryInfC{$A$}
	\RightLabel{\Tarrowelim}
	\BinaryInfC{$\bot$}
\RightLabel{\Tconjelim}
\BinaryInfC{$\bot$}
\RightLabel{\Tarrowintro}
\UnaryInfC{$\Tneg{A}$}
\RightLabel{\Tdisjintro}
\UnaryInfC{$\Tneg{A} \Tor \Tneg{\Tneg{A}}$}
\end{prooftree}
Now,
\vspace{\baselineskip}
\begin{deduction}
\AxiomC{}
\RightLabel{HE}
\UnaryInfC{$\Texists_{y} \Phi$}
	\AxiomC{}
	\RightLabel{DX}
	\UnaryInfC{$\Tforall_{x} \left(D{x} \Tor \Tneg{D{x}}\right)$}
	\RightLabel{\Tunivelim}
	\UnaryInfC{$D{y} \Tor \Tneg{D{y}}$}
		\AxiomC{}
		\UnaryInfC{Lemma 1}
			\AxiomC{}
			\UnaryInfC{Lemma 2}
	\RightLabel{\Tdisjelim}
	\TrinaryInfC{$\Tneg{A} \Tor \Tneg{\Tneg{A}}$}
\RightLabel{\Texistelim}
\BinaryInfC{$\Tneg{A} \Tor \Tneg{\Tneg{A}}$}
\end{deduction}
\end{proof}
\end{proposition}

\clearpage

\begin{proposition} \label{prop:GMP,DZ,DO,DX->WLEM}
$\GMPs, \TT/ \reduces \WLEMs$
\begin{proof}
Lemma 1:
\begin{prooftree}
\AxiomC{}
\UnaryInfC{$\Tforall_{x} \left(\left(D{x} \Tarrow \Tneg{\Tneg{A}}\right) \Tand \left(\Tneg{D{x}} \Tarrow \Tneg{A}\right)\right)$}
\RightLabel{\Tunivelim}
\UnaryInfC{$\left(D{0} \Tarrow \Tneg{\Tneg{A}}\right) \Tand \left(\Tneg{D{0}} \Tarrow \Tneg{A}\right)$}
	\AxiomC{}
	\UnaryInfC{$D{0} \Tarrow \Tneg{\Tneg{A}}$}
		\AxiomC{}
		\RightLabel{DZ}
		\UnaryInfC{$D{0}$}
	\RightLabel{\Tarrowelim}
	\BinaryInfC{$\Tneg{\Tneg{A}}$}
\RightLabel{\Tconjelim}
\BinaryInfC{$\Tneg{\Tneg{A}}$}
	\AxiomC{}
	\UnaryInfC{$\Tforall_{x} \left(\left(D{x} \Tarrow \Tneg{\Tneg{A}}\right) \Tand \left(\Tneg{D{x}} \Tarrow \Tneg{A}\right)\right)$}
	\RightLabel{\Tunivelim}
	\UnaryInfC{$\left(D{1} \Tarrow \Tneg{\Tneg{A}}\right) \Tand \left(\Tneg{D{1}} \Tarrow \Tneg{A}\right)$}
		\AxiomC{}
		\UnaryInfC{$\Tneg{D{1}} \Tarrow \Tneg{A}$}
			\AxiomC{}
			\RightLabel{DO}
			\UnaryInfC{$\Tneg{D{1}}$}
		\RightLabel{\Tarrowelim}
		\BinaryInfC{$\Tneg{A}$}
	\RightLabel{\Tconjelim}
	\BinaryInfC{$\Tneg{A}$}
\RightLabel{\Tarrowelim}
\BinaryInfC{$\bot$}
\RightLabel{\Tarrowintro}
\UnaryInfC{$\Tneg{\Tforall_{x} \left(\left(D{x} \Tarrow \Tneg{\Tneg{A}}\right) \Tand \left(\Tneg{D{x}} \Tarrow \Tneg{A}\right)\right)}$}
\end{prooftree}
Lemma 2:
\vspace{-\baselineskip}
\begin{prooftree}
\AxiomC{$\Tneg{\left(\left(D{x} \Tarrow \Tneg{\Tneg{A}}\right) \Tand \left(\Tneg{D{x}} \Tarrow \Tneg{A}\right)\right)}$}
	\AxiomC{}
	\UnaryInfC{$\Tneg{A}$}
		\AxiomC{}
		\UnaryInfC{$A$}
	\RightLabel{\Tarrowelim}
	\BinaryInfC{$\bot$}
	\RightLabel{\Tarrowintro}
	\UnaryInfC{$\Tneg{\Tneg{A}}$}
	\RightLabel{\Tarrowintro}
	\UnaryInfC{$D{x} \Tarrow \Tneg{\Tneg{A}}$}
		\AxiomC{}
		\UnaryInfC{$\Tneg{D{x}}$}
			\AxiomC{$D{x}$}
		\RightLabel{\Tarrowelim}
		\BinaryInfC{$\bot$}
		\RightLabel{\Tarrowintro}
		\UnaryInfC{$\Tneg{A}$}
		\RightLabel{\Tarrowintro}
		\UnaryInfC{$\Tneg{D{x}} \Tarrow \Tneg{A}$}
	\RightLabel{\Tconjintro}
	\BinaryInfC{$\left(D{x} \Tarrow \Tneg{\Tneg{A}}\right) \Tand \left(\Tneg{D{x}} \Tarrow \Tneg{A}\right)$}
\RightLabel{\Tarrowelim}
\BinaryInfC{$\bot$}
\RightLabel{\Tarrowintro}
\UnaryInfC{$\Tneg{A}$}
\RightLabel{\Tdisjintro}
\UnaryInfC{$\Tneg{A} \Tor \Tneg{\Tneg{A}}$}
\end{prooftree}
Lemma 3:
\vspace{-\baselineskip}
\begin{prooftree}
\AxiomC{$\Tneg{\left(\left(D{x} \Tarrow \Tneg{\Tneg{A}}\right) \Tand \left(\Tneg{D{x}} \Tarrow \Tneg{A}\right)\right)}$}
	\AxiomC{$\Tneg{D{x}}$}
		\AxiomC{}
		\UnaryInfC{$D{x}$}
	\RightLabel{\Tarrowelim}
	\BinaryInfC{$\bot$}
	\RightLabel{\Tarrowintro}
	\UnaryInfC{$\Tneg{\Tneg{A}}$}
	\RightLabel{\Tarrowintro}
	\UnaryInfC{$D{x} \Tarrow \Tneg{\Tneg{A}}$}
		\AxiomC{}
		\UnaryInfC{$\Tneg{A}$}
		\RightLabel{\Tarrowintro}
		\UnaryInfC{$\Tneg{D{x}} \Tarrow \Tneg{A}$}
	\RightLabel{\Tconjintro}
	\BinaryInfC{$\left(D{x} \Tarrow \Tneg{\Tneg{A}}\right) \Tand \left(\Tneg{D{x}} \Tarrow \Tneg{A}\right)$}
\RightLabel{\Tarrowelim}
\BinaryInfC{$\bot$}
\RightLabel{\Tarrowintro}
\UnaryInfC{$\Tneg{\Tneg{A}}$}
\RightLabel{\Tdisjintro}
\UnaryInfC{$\Tneg{A} \Tor \Tneg{\Tneg{A}}$}
\end{prooftree}
Now, where $\Phi = \Tneg{\Tforall_{x} \left(\left(D{x} \Tarrow
\Tneg{\Tneg{A}}\right) \Tand \left(\Tneg{D{x}} \Tarrow \Tneg{A}\right)\right)}
\Tarrow \Texists_{x} \Tneg{\left(\left(D{x} \Tarrow \Tneg{\Tneg{A}}\right) \Tand
\left(\Tneg{D{x}} \Tarrow \Tneg{A}\right)\right)}$,\\
\vspace{\baselineskip}
\begin{deduction}
\AxiomC{}
\RightLabel{GMP}
\UnaryInfC{$\Phi$}
	\AxiomC{}
	\UnaryInfC{Lemma 1}
\RightLabel{\Tarrowelim}
\BinaryInfC{$\Texists_{x} \Tneg{\left(\left(D{x} \Tarrow \Tneg{\Tneg{A}}\right) \Tand \left(\Tneg{D{x}} \Tarrow \Tneg{A}\right)\right)}$}
	\AxiomC{}
	\RightLabel{DX}
	\UnaryInfC{$\Tforall_{x} \left(D{x} \Tor \Tneg{D{x}}\right)$}
	\RightLabel{\Tunivelim}
	\UnaryInfC{$D{x} \Tor \Tneg{D{x}}$}
		\AxiomC{}
		\UnaryInfC{Lemma 2}
			\AxiomC{}
			\UnaryInfC{Lemma 3}
	\RightLabel{\Tdisjelim}
	\TrinaryInfC{$\Tneg{A} \Tor \Tneg{\Tneg{A}}$}
\RightLabel{\Texistelim}
\BinaryInfC{$\Tneg{A} \Tor \Tneg{\Tneg{A}}$}
\end{deduction}
\end{proof}
\end{proposition}

\clearpage

\begin{proposition} \label{prop:DP,LEM->GLPOA}
$\DPs, \LEMs \reduces \GLPOAs$
\begin{proof}
\begin{deduction}
\AxiomC{}
\RightLabel{DP}
\UnaryInfC{$\Texists_{y} \left(P{y} \Tarrow \Tforall_{x} P{x}\right)$}
	\AxiomC{}
	\RightLabel{LEM}
	\UnaryInfC{$P{y} \Tor \Tneg{P{y}}$}
		\AxiomC{}
		\UnaryInfC{$P{y} \Tarrow \Tforall_{x} P{x}$}
			\AxiomC{}
			\UnaryInfC{$P{y}$}
		\RightLabel{\Tarrowelim}
		\BinaryInfC{$\Tforall_{x} P{x}$}
		\RightLabel{\Tdisjintro}
		\UnaryInfC{$\Tforall_{x} P{x} \Tor \Texists_{x} \Tneg{P{x}}$}
			\AxiomC{}
			\UnaryInfC{$\Tneg{P{y}}$}
			\RightLabel{\Texistintro}
			\UnaryInfC{$\Texists_{x} \Tneg{P{x}}$}
			\RightLabel{\Tdisjintro}
			\UnaryInfC{$\Tforall_{x} P{x} \Tor \Texists_{x} \Tneg{P{x}}$}
	\RightLabel{\Tdisjelim}
	\TrinaryInfC{$\Tforall_{x} P{x} \Tor \Texists_{x} \Tneg{P{x}}$}
\RightLabel{\Texistelim}
\BinaryInfC{$\Tforall_{x} P{x} \Tor \Texists_{x} \Tneg{P{x}}$}
\end{deduction}
\end{proof}
\end{proposition}

\begin{proposition} \label{prop:DNSE,DZ,DO,DX->WLEM}
$\DNSEs, \TT/ \reduces \WLEMs$
\begin{proof}
Lemma 1:
\vspace{-\baselineskip}
\begin{prooftree}
\AxiomC{$\Tneg{\Texists_{x} \left(\left(D{x} \Tarrow \Tneg{\Tneg{A}}\right) \Tand \left(\Tneg{D{x}} \Tarrow \Tneg{A}\right)\right)}$}
	\AxiomC{}
	\UnaryInfC{$\left(D{0} \Tarrow \Tneg{\Tneg{A}}\right) \Tand \left(\Tneg{D{0}} \Tarrow \Tneg{A}\right)$}
	\RightLabel{\Texistintro}
	\UnaryInfC{$\Texists_{x} \left(\left(D{x} \Tarrow \Tneg{\Tneg{A}}\right) \Tand \left(\Tneg{D{x}} \Tarrow \Tneg{A}\right)\right)$}
\RightLabel{\Tarrowelim}
\BinaryInfC{$\bot$}
\RightLabel{\Tarrowintro}
\UnaryInfC{$\Tneg{\left(\left(D{0} \Tarrow \Tneg{\Tneg{A}}\right) \Tand \left(\Tneg{D{0}} \Tarrow \Tneg{A}\right)\right)}$}
	\AxiomC{}
	\UnaryInfC{$\Tneg{\Tneg{A}}$}
	\RightLabel{\Tarrowintro}
	\UnaryInfC{$D{0} \Tarrow \Tneg{\Tneg{A}}$}
		\AxiomC{}
		\UnaryInfC{$\Tneg{D{0}}$}
			\AxiomC{}
			\RightLabel{DZ}
			\UnaryInfC{$D{0}$}
		\RightLabel{\Tarrowelim}
		\BinaryInfC{$\bot$}
		\RightLabel{\Tarrowintro}
		\UnaryInfC{$\Tneg{A}$}
		\RightLabel{\Tarrowintro}
		\UnaryInfC{$\Tneg{D{0}} \Tarrow \Tneg{A}$}
	\RightLabel{\Tconjintro}
	\BinaryInfC{$\left(D{0} \Tarrow \Tneg{\Tneg{A}}\right) \Tand \left(\Tneg{D{0}} \Tarrow \Tneg{A}\right)$}
\RightLabel{\Tarrowelim}
\BinaryInfC{$\bot$}
\RightLabel{\Tarrowintro}
\UnaryInfC{$\Tneg{\Tneg{\Tneg{A}}}$}
\end{prooftree}
Lemma 2:
\vspace{-\baselineskip}
\begin{prooftree}
\AxiomC{$\Tneg{\Texists_{x} \left(\left(D{x} \Tarrow \Tneg{\Tneg{A}}\right) \Tand \left(\Tneg{D{x}} \Tarrow \Tneg{A}\right)\right)}$}
	\AxiomC{}
	\UnaryInfC{$\left(D{1} \Tarrow \Tneg{\Tneg{A}}\right) \Tand \left(\Tneg{D{1}} \Tarrow \Tneg{A}\right)$}
	\RightLabel{\Texistintro}
	\UnaryInfC{$\Texists_{x} \left(\left(D{x} \Tarrow \Tneg{\Tneg{A}}\right) \Tand \left(\Tneg{D{x}} \Tarrow \Tneg{A}\right)\right)$}
\RightLabel{\Tarrowelim}
\BinaryInfC{$\bot$}
\RightLabel{\Tarrowintro}
\UnaryInfC{$\Tneg{\left(\left(D{1} \Tarrow \Tneg{\Tneg{A}}\right) \Tand \left(\Tneg{D{1}} \Tarrow \Tneg{A}\right)\right)}$}
	\AxiomC{}
	\RightLabel{DO}
	\UnaryInfC{$\Tneg{D{1}}$}
		\AxiomC{}
		\UnaryInfC{$D{1}$}
	\RightLabel{\Tarrowelim}
	\BinaryInfC{$\bot$}
	\RightLabel{\Tarrowintro}
	\UnaryInfC{$\Tneg{\Tneg{A}}$}
	\RightLabel{\Tarrowintro}
	\UnaryInfC{$D{1} \Tarrow \Tneg{\Tneg{A}}$}
		\AxiomC{}
		\UnaryInfC{$\Tneg{A}$}
		\RightLabel{\Tarrowintro}
		\UnaryInfC{$\Tneg{D{1}} \Tarrow \Tneg{A}$}
	\RightLabel{\Tconjintro}
	\BinaryInfC{$\left(D{1} \Tarrow \Tneg{\Tneg{A}}\right) \Tand \left(\Tneg{D{1}} \Tarrow \Tneg{A}\right)$}
\RightLabel{\Tarrowelim}
\BinaryInfC{$\bot$}
\RightLabel{\Tarrowintro}
\UnaryInfC{$\Tneg{\Tneg{A}}$}
\end{prooftree}
Lemma 3:
\vspace{-\baselineskip}
\begin{prooftree}
\AxiomC{$\Tneg{\Tneg{\left(\left(D{x} \Tarrow \Tneg{\Tneg{A}}\right) \Tand \left(\Tneg{D{x}} \Tarrow \Tneg{A}\right)\right)}}$}
	\AxiomC{}
	\UnaryInfC{$\left(D{x} \Tarrow \Tneg{\Tneg{A}}\right) \Tand \left(\Tneg{D{x}} \Tarrow \Tneg{A}\right)$}
		\AxiomC{}
		\UnaryInfC{$D{x} \Tarrow \Tneg{\Tneg{A}}$}
			\AxiomC{$D{x}$}
		\RightLabel{\Tarrowelim}
		\BinaryInfC{$\Tneg{\Tneg{A}}$}
			\AxiomC{}
			\UnaryInfC{$\Tneg{A}$}
		\RightLabel{\Tarrowelim}
		\BinaryInfC{$\bot$}
	\RightLabel{\Tconjelim}
	\BinaryInfC{$\bot$}
	\RightLabel{\Tarrowintro}
	\UnaryInfC{$\Tneg{\left(\left(D{x} \Tarrow \Tneg{\Tneg{A}}\right) \Tand \left(\Tneg{D{x}} \Tarrow \Tneg{A}\right)\right)}$}
\RightLabel{\Tarrowelim}
\BinaryInfC{$\bot$}
\RightLabel{\Tarrowintro}
\UnaryInfC{$\Tneg{\Tneg{A}}$}
\RightLabel{\Tdisjintro}
\UnaryInfC{$\Tneg{A} \Tor \Tneg{\Tneg{A}}$}
\end{prooftree}

\clearpage

Lemma 4:
\vspace{-\baselineskip}
\begin{prooftree}
\AxiomC{$\Tneg{\Tneg{\left(\left(D{x} \Tarrow \Tneg{\Tneg{A}}\right) \Tand \left(\Tneg{D{x}} \Tarrow \Tneg{A}\right)\right)}}$}
	\AxiomC{}
	\UnaryInfC{$\left(D{x} \Tarrow \Tneg{\Tneg{A}}\right) \Tand \left(\Tneg{D{x}} \Tarrow \Tneg{A}\right)$}
		\AxiomC{}
		\UnaryInfC{$\Tneg{A}$}
			\AxiomC{}
			\UnaryInfC{$A$}
		\RightLabel{\Tarrowelim}
		\BinaryInfC{$\bot$}
		\RightLabel{\Tarrowintro}
		\UnaryInfC{$\Tneg{\Tneg{A}}$}
			\AxiomC{}
			\UnaryInfC{$\Tneg{D{x}} \Tarrow \Tneg{A}$}
				\AxiomC{$\Tneg{D{x}}$}
			\RightLabel{\Tarrowelim}
			\BinaryInfC{$\Tneg{A}$}
		\RightLabel{\Tarrowelim}
		\BinaryInfC{$\bot$}
	\RightLabel{\Tconjelim}
	\BinaryInfC{$\bot$}
	\RightLabel{\Tarrowintro}
	\UnaryInfC{$\Tneg{\left(\left(D{x} \Tarrow \Tneg{\Tneg{A}}\right) \Tand \left(\Tneg{D{x}} \Tarrow \Tneg{A}\right)\right)}$}
\RightLabel{\Tarrowelim}
\BinaryInfC{$\bot$}
\RightLabel{\Tarrowintro}
\UnaryInfC{$\Tneg{A}$}
\RightLabel{\Tdisjintro}
\UnaryInfC{$\Tneg{A} \Tor \Tneg{\Tneg{A}}$}
\end{prooftree}
Where $\Phi := \Tneg{\Tneg{\Texists_{x} \left(\left(D{x} \Tarrow
\Tneg{\Tneg{A}}\right) \Tand \left(\Tneg{D{x}} \Tarrow \Tneg{A}\right)\right)}}
\Tarrow \Texists_{x} \Tneg{\Tneg{\left(\left(D{x} \Tarrow \Tneg{\Tneg{A}}\right)
\Tand \left(\Tneg{D{x}} \Tarrow \Tneg{A}\right)\right)}}$,
\vspace{1.5\baselineskip}
\begin{deduction}
\AxiomC{}
\RightLabel{DNSE}
\UnaryInfC{$\Phi$}
	\AxiomC{}
	\UnaryInfC{Lemma 1}
		\AxiomC{}
		\UnaryInfC{Lemma 2}
	\RightLabel{\Tarrowelim}
	\BinaryInfC{$\bot$}
	\RightLabel{\Tarrowintro}
	\UnaryInfC{$\Tneg{\Tneg{\Texists_{x} \left(\left(D{x} \Tarrow \Tneg{\Tneg{A}}\right) \Tand \left(\Tneg{D{x}} \Tarrow \Tneg{A}\right)\right)}}$}
\RightLabel{\Tarrowelim}
\BinaryInfC{$\Texists_{x} \Tneg{\Tneg{\left(\left(D{x} \Tarrow \Tneg{\Tneg{A}}\right) \Tand \left(\Tneg{D{x}} \Tarrow \Tneg{A}\right)\right)}}$}
	\AxiomC{}
	\RightLabel{DX}
	\UnaryInfC{$\Tforall_{x} \left(D{x} \Tor \Tneg{D{x}}\right)$}
	\RightLabel{\Tunivelim}
	\UnaryInfC{$D{x} \Tor \Tneg{D{x}}$}
		\AxiomC{}
		\UnaryInfC{Lemma 3}
			\AxiomC{}
			\UnaryInfC{Lemma 4}
	\RightLabel{\Tdisjelim}
	\TrinaryInfC{$\Tneg{A} \Tor \Tneg{\Tneg{A}}$}
\RightLabel{\Texistelim}
\BinaryInfC{$\Tneg{A} \Tor \Tneg{\Tneg{A}}$}
\end{deduction}
\end{proof}
\end{proposition}